\documentclass[11pt, twoside, reqno]{amsart}
\usepackage{amsmath, amssymb, amsfonts, amsthm}
\usepackage{graphicx}
\usepackage{tikz}
\usepackage[a4paper, margin = 2.8cm]{geometry}
\usepackage[english]{babel}
\usepackage{enumerate}
\usepackage[
colorlinks = true, 
linkcolor = blue, 
citecolor = blue, 
urlcolor = blue, 
pagebackref = true
]{hyperref}

\newtheorem{theorem}{Theorem}[section]
\newtheorem{lemma}[theorem]{Lemma}
\newtheorem{proposition}[theorem]{Proposition}
\newtheorem{corollary}[theorem]{Corollary}
\newtheorem{remark}[theorem]{Remark}
\newtheorem{question}[theorem]{Question}
\newtheorem{conjecture}{Conjecture}
\newtheorem*{conjecture*}{Conjecture}

\newtheorem*{rmks}{Remark}
\newtheorem*{question*}{Question}
\newcommand{\R}{\mathbb{R}}
\numberwithin{equation}{section}

\makeatletter
\@namedef{subjclassname@2020}{Mathematics Subject Classification 2020}
\makeatother

\title[On location of fail points]{On the location of the maximal gradient of the torsion function over some non-symmetric planar domains}

\author{Qinfeng Li}
\address{School of Mathematics, Hunan University, Changsha, P.R. China.}
\email{liqinfeng1989@gmail.com; liqinfeng@hnu.edu.cn}

\author{Shuangquan Xie}
\address{School of Mathematics, Hunan University, Changsha, P.R. China.}
\email{xieshuangquan@hnu.edu.cn}

\author{Hang Yang}
\address{School of Mathematics, Hunan University, Changsha, P.R. China.}
\email{hangyang0925@gmail.com}

\author{Ruofei Yao}
\address{School of Mathematics, South China University of Technology, Guangzhou, P.R. China.}
\email{yaorf5812@126.com; ruofeiyaopde@gmail.com}


\begin{document}

\begin{abstract}
We investigate the location of the maximal gradient of the torsion function on certain non-symmetric planar domains. First, by establishing uniform estimates for convex narrow domains, we show that as a planar domain bounded by two graphs becomes increasingly narrow, the location of the maximal gradient of its torsion function converges to the endpoints of the longest vertical segment, with smaller curvature among them. This result confirms that Saint-Venant's conjecture on the location of fail points holds for asymptotically narrow domains. Second, for triangles, we prove that the maximal gradient of the torsion function always occurs on the longest side, lying between the foot of the altitude and the midpoint of that side. Moreover, via nodal line analysis, we show that, restricted to each side, the critical point of the gradient is unique and non-degenerate. Additionally, by perturbation and barrier arguments, we establish that for a class of nearly equilateral triangles, this critical point is closer to the midpoint than to the foot of the altitude, and the maximal gradient at the midpoint exceeds that at the foot of the altitude. Third, employing the reflection method, we demonstrate that for a non-concentric annulus, the maximal gradient of the torsion function is always attained at the point on the inner boundary closest to the center of the outer boundary.
\end{abstract}

\date{\today}
\maketitle


\section{Introduction}
\subsection{Background}
This paper investigates one of the most fundamental objects in mathematical physics: the torsion function, defined as the unique solution \(u_{\Omega}\) to
\begin{equation} \label{eq102Torsion}
\begin{cases}
- \Delta u = 1 &\text{in } \Omega, \\
u = 0 &\text{on } \partial \Omega.
\end{cases}
\end{equation}
Here, $\Omega$ is often assumed to be a planar Lipschitz domain. 
The terminology ``torsion function'' goes back to Saint-Venant's elasticity theory \cite{SV1856}; see also \cite{Flu63} for additional physical background.
Although \eqref{eq102Torsion} is elementary, the torsion function has become a major research topic due to its deep connections with several areas of mathematics and its relevance in applications such as mechanical engineering and fluid mechanics. The literature on the subject continues to expand. In recent years, torsion functions have attracted considerable attention not only because of their central role in geometric functional inequalities and shape optimization (see, e.g., the monographs \cite{Hen06, Hen17} and the recent work \cite{ANT23}), but also because they serve as effective predictors for the localization of extreme values of Laplacian eigenfunctions, which has led to a number of notable results \cite{ADJ16, Bec20, FM12, RS18, Ste17, Ste22}. 
More recently, the maximization of the shape functionals \(\lvert \nabla u_{\Omega} \rvert / \sqrt{\lvert \Omega \rvert}\) and \(\lvert \nabla u_{\Omega} \rvert / \lvert \partial \Omega \rvert\) has been studied in \cite{BFM25, Hua25}, where existence results as well as structural and geometric properties of optimizers are established. 
We also refer to \cite{KM93} for a survey of qualitative properties of torsion functions on convex domains.

The present paper is motivated by the problem of determining where \(\lvert \nabla u_{\Omega} \rvert\) attains its maximum on \(\overline{\Omega}\). This problem traces back to Saint-Venant's conjecture on fail points. In Saint-Venant's elasticity theory, $\Omega$ represents the cross section of an elastic cylindrical bar, and $|\nabla {u}_{\Omega}|$ measures the shear stress. Points at which $|\nabla {u}_{\Omega}|$ attains its maximum are called \emph{fail points}, as the bar transitions from an elastic to a plastic state when $|\nabla {u}_{\Omega}|$ becomes too large. Since $|\nabla {u}_{\Omega}|^2$ is subharmonic, fail points must lie on the boundary of $\Omega$. Saint-Venant made the following conjecture: 

\begin{conjecture}[Saint-Venant's Conjecture]
For a convex planar domain that is symmetric about the two coordinate axes, fail points occur at the contact points of the largest inscribed circle.
\end{conjecture}

Kawohl \cite{Kaw87} proved Saint-Venant's conjecture under the additional assumption that the curvature on $\partial\Omega$ is monotonic in the first quadrant. Under this condition, the norm of the gradient of the torsion function is also monotonic in the first quadrant. Later, Ramaswamy \cite{Ram90} and Sweers \cite{Swe89, Swe92} disproved Saint-Venant's conjecture by constructing domains symmetric about two axes such that the long and short axes are of equal length but the endpoints have different curvatures. Ramaswamy subsequently proposed a modified conjecture \cite{Ram90}, which has been widely accepted in the mechanical engineering community: 

\begin{conjecture}[Ramaswamy's Conjecture]
For a planar convex domain that is symmetric about the two coordinate axes, the fail points occur either at the contact points of the largest inscribed circle or at points of minimal curvature. 
\end{conjecture}

Recently, the first and fourth authors \cite{LY24} disproved Ramaswamy's conjecture by providing a precise description of the location of fail points for nearly circular domains. Their results indicate that the location of fail points actually depends on nonlocal features of the domain, whereas curvature and contact points of the largest inscribed circle are local characteristics. 

Finding the location of fail points is also relevant to thermal problems. Indeed, the solution $u$ to \eqref{eq102Torsion} can be interpreted as the steady-state temperature within a thermal body $\Omega$, where the heat source is uniformly distributed (with unit intensity), the heat transfer occurs solely through conduction, and the external temperature remains constant, specifically $0$. Consequently, the quantity 
\begin{equation*}
\max_{\overline{\Omega}}|\nabla u| = \max_{\partial \Omega}|\nabla u|
\end{equation*}
represents the maximum temperature flux across the boundary. Thus, the location of fail points identifies where this maximal flux occurs, which is of both theoretical and practical interest. Moreover, in the insightful work \cite{But85}, the author introduces a thin insulation problem, aiming to find the optimal distribution of insulating material to maximize the average temperature in a given domain. This problem has recently attracted renewed attention, as seen in \cite{BBN17}, \cite{DNST21}, \cite{HLL22}, \cite{HLL24}, etc. For the insulation model with a uniform heat source, it is shown in \cite{HLL24} that, for any $C^2$ domain that is not a ball, if the total amount of insulating material exceeds a certain threshold, then the optimal insulation necessitates covering the entire
boundary. When the amount is below this threshold, it is advantageous not to cover the entire boundary. The $C^2$ regularity assumption was recently removed by Figalli and Zhang, who established rigidity for Serrin's overdetermined system in rough domains \cite{FZ26}. Notably, an earlier result in \cite{ER03} shows that, as the total amount of insulating material tends to zero, the optimal distribution concentrates at boundary points of maximal temperature flux for the solution to \eqref{eq102Torsion}, that is, at the fail points. Hence, determining the location of fail points provides valuable insight for designing optimal insulation strategies.

For further historical remarks on the location of fail points, we refer to \cite{Flu63}, \cite{HS21}, \cite{Kaw87}, and references therein. It is also an important problem to determine how large the maximal gradient of the torsion function can be over all convex planar domains of fixed area, due to its relevance in various applications. In this direction, we refer to the work of Hoskins and Steinerberger \cite{HS21}, which, to our knowledge, provides so far the best upper bound for the maximal gradient. 

Despite the significant progress that has been made on the location of fail points since Saint-Venant's conjecture, many questions remain open. First, most previous research has focused on convex domains with one or two axes of symmetry, leaving domains without symmetry or non-convex domains largely unexplored. Moreover, Saint-Venant's original conjecture did not involve curvature. His intuition may have been that boundary points closer to the ``center" (i.e., the point where ${u}_{\Omega}$ attains its maximum) exhibit larger values of $|\nabla {u}_{\Omega}|$. Although Saint-Venant's conjecture has been disproved, the partial validity of his intuition has received limited attention since Kawohl's work \cite{Kaw87}, as much of the literature has focused on counterexamples. Thus, the underlying spirit of Saint-Venant's conjecture remains insufficiently explored. Whether his intuition can be rigorously formulated under reasonable assumptions, without explicit curvature information, remains an open and natural question. 

With these considerations, we study the location of fail points in three distinct classes of domains: narrow domains, triangles, and non-concentric annuli. Further motivations and our main results are detailed below.

\subsection{Our results on narrow domains}

Motivated by understanding when Saint-Venant's original conjecture can hold without imposing curvature assumptions, we begin by investigating the location of fail points on asymptotically narrow domains. Intuitively, the narrow-domain regime most closely reflects Saint-Venant's original viewpoint. Our goal is to show that, as the domain becomes increasingly narrow, fail points concentrate near the location predicted by Saint-Venant's conjecture.

We consider a family of narrow planar domains
\begin{equation}\label{eq106thinDomain}
\Omega_{\epsilon}: = \{(x, y) \in \mathbb{R}^2: \epsilon f_{1}(x) < y < \epsilon f_{2}(x), \, x\in(a, b)\}
\end{equation}
for $\epsilon > 0$, where $f_{1}, f_{2} \in C([a, b])$ satisfy 
\begin{equation}\label{eq107thinf}
f_{1}(a) = f_{2}(a) = f_{1}(b) = f_{2}(b) = 0.
\end{equation}
and $f_{2}(x) > f_{1}(x)$ for all $x \in (a, b)$. Let $u_{\epsilon}$ denote the torsion function in $\Omega_{\epsilon}$, i.e., the solution of
\[
-\Delta u_{\epsilon} = 1 \ \text{in }\Omega_{\epsilon}, \qquad u_{\epsilon} = 0 \ \text{on }\partial\Omega_{\epsilon}.
\]
By a \emph{fail point} on $\partial\Omega_{\epsilon}$ we mean a point where the boundary quantity $|\nabla u_{\epsilon}|$ attains its maximum.

\begin{theorem}\label{thm12narrow}
Let $a < b$ be fixed real numbers. 
Let $f_{1}, f_{2} \in C^2([a, b])$ satisfy~\eqref{eq107thinf}, with \(f_{1}\) convex and \(f_{2}\) concave. Let $\Omega_{\epsilon}$ be the domain defined in~\eqref{eq106thinDomain}. 

Let $I \subset (a, b)$ be the (possibly degenerate) set on which $f_{2} - f_{1}$ achieves its maximum. 
Then the distance between the $x$-coordinates of fail points on $\partial\Omega_{\epsilon}$ and the set $I$ converges to $0$ as $\epsilon \to 0^{+}$. 
\end{theorem}


Theorem~\ref{thm12narrow} only yields asymptotic information on the $x$-coordinates of fail points, which is governed by the leading $\epsilon^{2}$ term of the following pointwise expansion for the boundary gradient of $u_\epsilon$:
\begin{align}
\label{1.1expansion}
|\nabla u_\epsilon(x, f_k(x))|^2=\frac{1}{4}(f_{1}(x) - f_{2}(x))^{2}\epsilon^{2} + O(\epsilon^{4}), \quad k=1,2.
\end{align}

To further distinguish the two boundary components, one needs the precise coefficient of the next-order term. For simplicity, we consider the case where $I$ reduces to a single point, and we have the following result.

\begin{theorem}\label{thm13Location}
Under the assumptions of Theorem~\ref{thm12narrow}, assume that $f_{2} - f_{1}$ has a unique maximizer $x = z_{0} \in (a, b)$. 
If the curvature of $\partial\Omega_{1}$ at $(z_{0}, f_{1}(z_{0}))$ is strictly smaller than that at $(z_{0}, f_{2}(z_{0}))$, then there exists $\epsilon_{0} > 0$ such that for all $0 <\epsilon < \epsilon_{0}$, 
\begin{equation*}
|\nabla u_{\epsilon}(z_{0}, \epsilon f_{1}(z_{0}))| > |\nabla u_{\epsilon}(z_{0}, \epsilon f_{2}(z_{0}))|.
\end{equation*}
In particular, for sufficiently small $\epsilon$, the fail points always lie on the graph of $\epsilon f_1$ in this situation.
\end{theorem}

For example, in Figure~\ref{fig11narrow}, fail points on $\partial\Omega_{\epsilon}$ are close to the point $P_{\epsilon}$ when $\epsilon$ is small.

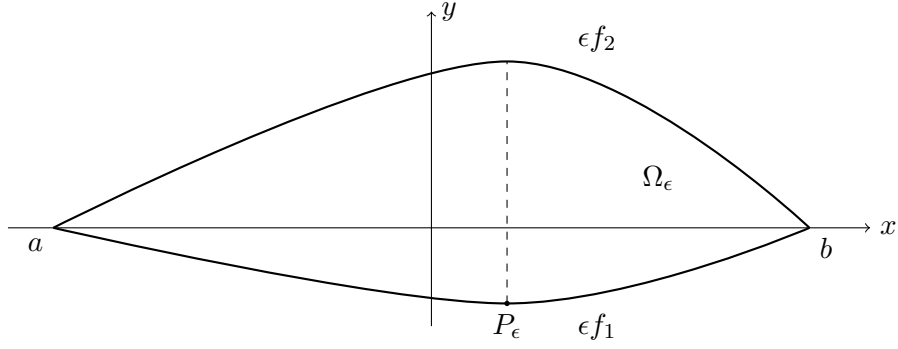
\begin{figure}[htp]
\centering
\begin{tikzpicture}[scale = 2]
\pgfmathsetmacro\leftx{ - 3}; 
\pgfmathsetmacro\rightx{2}; 
\pgfmathsetmacro\upy{1.1}; 
\pgfmathsetmacro\belowy{ - 0.5}; 
\node[below left] at (\leftx, 0) {$a$}; 
\node[below right] at (\rightx, 0) {$b$}; 
\node[above right] at (0.2*\rightx, \upy) {$\epsilon f_{2}$}; 
\node[below right] at (0.2*\rightx, \belowy) {$\epsilon f_{1}$}; 
\node[below] at (0, \belowy) {$P_{\epsilon}$}; 
        \fill (0,\belowy) circle (0.5pt);
\node at (\rightx*0.5, \upy*0.3) {$\Omega_{\epsilon}$}; 
        \draw[thick] plot [smooth, tension = 0.6] 
coordinates {(\leftx, 0) (0, \upy) (\rightx, 0) }; 
\draw[thick] plot [smooth, tension = 0.6] 
coordinates {(\rightx, 0) (0, \belowy) (\leftx, 0)}; 
\draw[ -> ] (1.1*\leftx, 0) -- (1.2*\rightx, 0) node[right]{$x$}; 
\draw[dashed ] (0, \belowy) -- (0, \upy, 0);
\draw[ -> ] ({(\leftx + \rightx)/2}, 1.3*\belowy) -- ({(\leftx + \rightx)/2}, 1.3*\upy, 0) node[right]{$y$}; 
\end{tikzpicture}
\caption{Planar domain bounded by two graphs}
\label{fig11narrow}
\end{figure}

A key technical ingredient of proving Theorem \ref{thm13Location} requires the $C^4$ regularity of functions $f_1$ and $f_2$, which enables us to derive the following more precise pointwise asymptotic expansion formula: for $k = 1, 2$ and $x\in(a, b)$, 
\begin{align}\label{eq104Expansion}
|\nabla u_{\epsilon}(x, \epsilon f_{k}(x))|^{2}
= \frac{1}{4}\bigl(f_{2}(x) - f_{1}(x)\bigr)^{2}\epsilon^{2}
+\epsilon^{4}\lambda_{k, 2}(x)
+O(\epsilon^{6}), 
\end{align}
as $\epsilon\to 0^{+}$, where $\lambda_{k,2}$ depends on derivatives of $f_{1}$ and $f_{2}$ up to second order; see \eqref{lambdak2} below for its explicit expression. 
The expansion \eqref{eq104Expansion} is obtained by a formal Taylor expansion of the rescaled function
\begin{equation*}
\tilde u(\epsilon, x, y): = u_{\epsilon}(x, \epsilon y)\quad \text{on }\Omega_{1}, 
\end{equation*}
namely, 
\begin{equation*}
\tilde u(\epsilon, x, y) = \sum_{j = 0}^{4}\tilde u_{j}(x, y)\epsilon^{j}+\tilde R_{4}(\epsilon, x, y).
\end{equation*}
A rigorous justification requires appropriate estimate for the remainder term $\tilde R_{4}$, where we need the technical assumption of the $C^4$ regularity of $f_1$ and $f_2$. The main difficulty is that the equation satisfied by $\tilde R_{4}$ is not uniformly elliptic in $\Omega_1$ as $\epsilon\to 0$. To overcome this, instead of seeking the precise asymptotic order of $\|\tilde R_{4}(\epsilon, \cdot)\|_{W^{1, \infty}(\Omega_1)}$, we directly estimate $R_4(\epsilon,x,y):=\tilde R_{4}(\epsilon, x, y/\epsilon)$ defined in $\Omega_\epsilon$, via a uniform $L^\infty$ estimate of sharp order for the torsion function and its gradient over $\Omega_\epsilon$; see Section~\ref{Sect2Narrow} for details.

Theorems~\ref{thm12narrow} and \ref{thm13Location} indicate that, as a planar convex domain bounded by two graphs becomes increasingly narrow, the leading-order determinant of the fail-point location is the local thickness $f_{2}-f_{1}$ (the $\epsilon^{2}$ term in \eqref{eq104Expansion}), while curvature information enters at the next order (the $\epsilon^{4}$ term). This not only realize the spirit of Saint-Venant, but also supports the following commonly cited engineering folklore for narrow domains (see \cite{Flu63}): ``the fail point occurs near one of the points of contact of the largest inscribed circle, and of these at that one where the boundary is least convex or most concave.''


In Theorem \ref{thm12narrow} we assumed $f_{1}, f_{2}\in C^{2}([a, b])$. The case where $f_i'(x)$ diverges as $x\to a$ or $x\to b$ (which occurs when a support line is perpendicular to the $x$-axis) will also be examined in Section~\ref{Sect2Narrow}, via a truncation argument and comparison with narrow rectangles.


\subsection{Our results on triangles}
We consider the location of fail points on triangular regions, which typically represent a class of domains that may not be symmetric about any axis. The reason for studying fail points on triangles is also motivated by the recent advances on the hot spot conjecture, whose theme is to determine the location of extremal points for the second eigenfunction of the Neumann-Laplacian operator. It is one of the most important locating problems in the field of partial differential equations, and most research works on this conjecture have been centered around planar graphs. Notably, groundbreaking studies on triangles, as seen in \cite{AB02, AB04, BB99, CGY26, JM20, JM22, NSY20, Roh21, Siu15} and references therein, have provided valuable insights into this conjecture. Drawing inspiration from these elegant findings, we delve into finding the location of fail points on triangles. 

First, via the moving plane method or reflection method, we first obtain the following result.

\begin{theorem}\label{thm14LongestSide}
For any triangle \(\triangle_{ABC}\), fail points can occur only on the longest side \(\Gamma\). Moreover, every fail point lies on the subsegment of \(\Gamma\) between the midpoint of \(\Gamma\) and the foot of the altitude from the opposite vertex onto \(\Gamma\).
\end{theorem}

A natural question that follows is to count the number of fail points on the boundary of a domain $\Omega$, or more generally, number of critical points of the function: 
\begin{align*}
&\partial \Omega \rightarrow \mathbb{R}, \\
&x \mapsto |\nabla u|^2(x), 
\end{align*}where $u$ is the torsion function on $\Omega$. In this direction and on triangular domains, we obtain the following result, via nodal line analysis. The proof is inspired by \cite{CGY26}, \cite{JN00} and \cite{JM20}.

\begin{theorem} \label{thm15Uniqueness}
Let $u$ be the torsion function in a triangle $\Omega = \triangle_{ABC}$. Then the critical point of the function 
$x \in \partial \Omega \mapsto |\nabla u|^2(x)$ on each side of $\triangle_{ABC}$ is unique and nondegenerate. As a result, on each side of the triangle $\triangle_{ABC}$, there is at most one fail point.
\end{theorem}

\begin{rmks}
By regularity theory, the torsion function on a triangle has a \(C^{1}\) extension to the closure of the triangle, and its gradient vanishes at the vertices. Hence, by the theorem above, on each side of the triangle the function \(s \mapsto |\nabla u(s)|\) attains its unique maximum at a single critical point.
\end{rmks}

Combining Theorem \ref{thm14LongestSide} and Theorem \ref{thm15Uniqueness}, we know that when restricted to each side, the maximal point of the norm of the gradient of the torsion function is unique, lying between the foot of the altitude and the midpoint of the side. When restricted to a given side of a triangle, since the midpoint and the foot of the perpendicular of that side are two landmark points, and also in order to obtain a more precise interval estimation of fail points on triangles, it is natural to raise the following two questions:

\begin{question}\label{OpenQue16}
On each side of a given triangle, which is closer to the maximal point of the norm of the gradient of the torsion function, the foot of the altitude or the midpoint?
\end{question}

\begin{question}\label{OpenQue17}
On each side of a given triangle, at which of the two points (the foot of the altitude or the midpoint of the side) does the magnitude of the norm of the gradient of the torsion function have a larger value?
\end{question}

Since the maximum length of a perpendicular segment drawn from a point on the longest side to the opposite boundary is attained precisely at the foot of the altitude, by Theorem \ref{thm12narrow} (or \eqref{1.1expansion}) and up to a truncation argument (see the proof of Theorem \ref{thm22NarrowBlow}), it is expected that for sufficiently narrow non-isosceles triangles, the unique fail point (on the longest side) should be closer to the foot of the altitude, and the norm of the gradient of the torsion function takes a larger value at the foot than at the midpoint.  One may wonder whether it is always the case for all non-isosceles triangles. 

In the following theorem, we give a negative answer for a kind of perturbed equilateral triangles.

\begin{theorem} \label{thm18ctm}
Let $\{\Omega_{t}\}_{t \geq 0}$ be the family of triangles with vertices $A_{t} = ( - \sqrt{3}/3, 0)$, $B_{t} = (\sqrt{3}/3 + t, 0)$ and $C_{t} = (0, 1)$, and let $u(t; x, y)$ denote the torsion function on $\Omega_{t}$. 
Then, for $t \to 0^{ + }$, the unique fail point of $\Omega_{t}$ is given by $p_{t} = (x(t), 0)$ where 
\begin{align}\label{estimatedinterval}
\frac{7}{24}t< x(t) < \frac{12}{24}t.
\end{align}
Moreover, 
\begin{align}\label{compare1}
u_{y} \Bigl(t; \frac{t}{2}, 0\Bigr) - u_{y} (t; 0, 0) > \frac{1}{32}t^2.
\end{align}
\end{theorem}

\begin{rmks}
Triangles $\Omega_{t}$ above are generated by stretching the base of the equilateral triangle $\Omega_{0}$, such that $A_{t}B_{t}$ is always the longest side(see Figure \ref{fig12ctm} below). The foot of the altitude and the midpoint of $A_{t}B_{t}$ are $(0, 0)$ and $(t/2, 0)$ respectively, and thus for $t > 0$ small, $|x(t) - t/2| < |x(t) - 0|$, due to \eqref{estimatedinterval}. Therefore, the theorem above implies that the fail point is closer to the midpoint of the longest side than to the foot of the altitude. Also, \eqref{compare1} implies that the norm of the gradient has a larger value at the midpoint.
\end{rmks}

\begin{figure}[htp]
\centering
\begin{tikzpicture}[scale = 2.5]
\pgfmathsetmacro\L{1/sqrt(3)}; 
\pgfmathsetmacro\h{1}; 
\pgfmathsetmacro\t{0.4}; 

\fill ( - \L, 0) circle (0.02 ) node[below ] {\small$A_{t}$}; 
\fill (\L, 0) circle (0.02 ) node[below ] {\small$B_{0}$}; 
\fill (\L + \t, 0) circle (0.02 ) node[below right] {\small$B_{t}$}; 
\fill (0, \h) circle (0.02 ) node[above right] {\small$C_{t}$}; 
\fill (0, 0) circle (0.02 ) node[below right] {\small$O$}; 

\draw[ -> ] ( - 0.8, 0) -- (1.5, 0) node[right] {$x$}; 
\draw[ -> ] (0, - 0.2) -- (0, 1.3) node[right] {$y$}; 

\draw ( - \L, 0) -- (\L + \t, 0); 
\draw (\L + \t, 0) -- (0, \h); 
\draw (0, \h) -- ( - \L, 0); 
\draw[thick, dotted] (0, \h) -- (\L, 0); 

\node[above right] at (0.6, 0.5*\h) {\small$\Omega^{t}$}; 
\end{tikzpicture}
\caption{Stretched equilateral triangles}
\label{fig12ctm}
\end{figure}
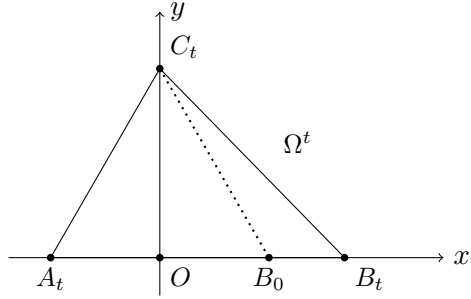

Therefore, in view of Theorem \ref{thm12narrow} and Theorem \ref{thm18ctm}, the answer to Questions \ref{OpenQue16}-\ref{OpenQue17} should vary depending on specific shape of a given triangle, at least restricted to the longest side. In general, these two questions are extremely delicate. Even for the special case considered in Theorem \ref{thm18ctm}, the proof of the estimates necessitate delicate barrier function constructions and comparisons of second order mixed derivatives, see Section~\ref{Sect4EquTri}. Based on numerical results, we have the following conjecture.

\begin{conjecture}
Let \(\Omega\) be a non-isosceles, non-acute triangle. Then the fail point lies on the longest side and is always closer to the foot of the altitude onto that side than to the midpoint of the side.
\end{conjecture}


\subsection{Our results on non-concentric annuli}
The third category of domains we consider comprises non-concentric annuli, which serve as typical examples of domains containing holes. Although an explicit formula for the torsion function can be derived in the case of a concentric annulus due to its radial symmetry, no such explicit formula exists for a non-concentric annulus. Through the reflection method, we establish the following theorem.

\begin{theorem} \label{thm19Annulus}
Suppose that $\Omega\subset\mathbb{R}^{2}$ is a non-concentric annulus, then the fail point in $\overline{\Omega}$ is always located at the point on the inner ring closest to the center of the outer ring.
\end{theorem}

\subsection{Outline of the paper}
In Section~\ref{Sect2Narrow}, we consider location of fail points on convex narrow domains and we will mainly prove Theorem \ref{thm12narrow} and Theorem \ref{thm13Location}. In Section~\ref{Sect3Triangle}, we consider location and uniqueness of fail points on triangles, and we will mainly prove Theorem \ref{thm14LongestSide} and Theorem \ref{thm15Uniqueness}. In Section~\ref{Sect4EquTri}, we address Questions \ref{OpenQue16}-\ref{OpenQue17} on two classes of nearly equilateral triangles, and we will prove Theorem \ref{thm18ctm} and related results. In Section~\ref{Sect5annuli}, we consider location of fail points on non-concentric annuli and we will prove Theorem \ref{thm19Annulus}.


\section{Fail points on narrow domains} \label{Sect2Narrow}

In this section, we derive asymptotic formulas for the boundary gradient of the torsion function $u_{\epsilon}$ on the narrow domain $\Omega_{\epsilon}$: 
\begin{equation}
\Omega_{\epsilon}: = \{(x, y) \in \mathbb{R}^2: \epsilon f_{1}(x) < y < \epsilon f_{2}(x), \, x\in(a, b)\}
\end{equation}
where $\epsilon > 0$, $f_{2} > f_{1}$ in $(a, b)$, and 
\begin{equation}\label{eq202thinf}
f_{k}(a) = f_{k}(b) = 0, \qquad k = 1, 2.
\end{equation}
    Throughout, \(g(\epsilon,x)=O(\epsilon^{k})\) as \(\epsilon\to0^{+}\) means that \(|g(\epsilon,x)|/\epsilon^{k}\) is uniformly bounded (in \(x\)) for all sufficiently small \(\epsilon>0\).

We begin with a useful $L^{\infty}$ bound for the torsion function $u_{\epsilon}$ and its gradient norm on \(\Omega_{\epsilon}\).

\begin{lemma} \label{lma21Estimate}
Let $f_{1}, f_{2} \in C^{2}([a, b])$ with $f_{2} > f_{1}$ in $(a, b)$, $f_1$ convex and $f_2$ concave. Then, the torsion function $u_{\epsilon}$ on $\Omega_{\epsilon}$ satisfies 
\begin{equation*}
\|u_{\epsilon}\|_{L^{\infty}(\Omega_{\epsilon})} = O(\epsilon^{2}).
\end{equation*}
and
\begin{equation*}
\|\nabla u_{\epsilon}\|_{L^\infty(\partial \Omega_\epsilon)} = \|\nabla u_{\epsilon}\|_{L^{\infty}(\Omega_{\epsilon})} = O(\epsilon). 
\end{equation*}
\end{lemma}

\begin{proof}
We set $H_{ - }: = \min_{x \in [a, b]} f_{1}(x)$ and $H_{ + } = \max_{x \in [a, b]} f_{2}(x)$. Define an auxiliary function 
\begin{equation*}
v_{\epsilon}(x, y) = \frac{1}{2}(\epsilon H_{ + } - y)(y - \epsilon H_{ - }).
\end{equation*}
A direct computation gives 
\begin{align*}
\begin{cases}
- \Delta v_{\epsilon} = 1, \quad &\text{in $\Omega_{\epsilon}$}\\
v_{\epsilon} \geq 0 &\text{on $\partial \Omega_{\epsilon}$}
\end{cases}
\end{align*}
Applying the maximum principle to $v_{\epsilon} - u_{\epsilon}$, we conclude that \(v_{\epsilon} - u_{\epsilon} \geq 0\) in \(\Omega_{\epsilon}\). Consequently, 
\begin{equation*}
\|u_{\epsilon}\|_{L^{\infty}(\Omega_{\epsilon})} \leq \|v_{\epsilon}\|_{L^{\infty}(\Omega_{\epsilon})} \leq \frac{1}{8}\epsilon^{2}(H_{ + } - H_{ - })^{2}=O(\epsilon^2). 
\end{equation*}

Using the $C^{1}$ regularity of $u_{\epsilon}$ over $\overline{\Omega}_{\epsilon}$ (see, for example, ~\cite{Gri11}) and the classical inequality \cite[Eq.~6.12]{Spe81} (see also \cite{HS21}) for planar convex domains, we obtain the uniform upper bound for $\|\nabla u_{\epsilon}\|_{L^{\infty}(\Omega_{\epsilon})}$ as follows: 
\begin{equation*}
\|\nabla u_{\epsilon}\|_{L^{\infty}(\partial\Omega_{\epsilon})} = \|\nabla u_{\epsilon}\|_{L^{\infty}(\Omega_{\epsilon})} < 2\|u_{\epsilon}\|^{1/2}_{L^{\infty}(\Omega_{\epsilon})} = O(\epsilon). 
\end{equation*}
    This concludes the proof. 
\end{proof}


\begin{proof}[Proof of Theorem~\ref{thm12narrow}]
Let $$\phi_\epsilon(x,y)=-\frac{1}{2}\left(y-\epsilon f_1(x)\right) \left(y-\epsilon f_2(x)\right).$$
Then \begin{align*}
\begin{cases}
-\Delta \phi_\epsilon=1+O(\epsilon^2)\quad &\mbox{in $\Omega$}\\
\phi_\epsilon=0\quad &\mbox{on $\partial \Omega$}
.\end{cases}
\end{align*}
Moreover, 
\begin{align*}
|\nabla \phi_\epsilon(x,\epsilon f_k(x))|^2=\frac{1}{4}(f_{1}(x) - f_{2}(x))^{2}\epsilon^{2} + O(\epsilon^{3}), \quad k=1,2.
\end{align*}
Let $w_\epsilon=u_\epsilon-\phi_\epsilon$. Then
 \begin{align*}
\begin{cases}
\Delta w_\epsilon=O(\epsilon^2)\quad &\mbox{in $\Omega$}\\
w_\epsilon=0\quad &\mbox{on $\partial \Omega$}
.\end{cases}
\end{align*}
There exists a universal constant $M>0$ such that $\epsilon^2 Mu_\epsilon\pm w_\epsilon$ is superharmonic in $\Omega_\epsilon$. By the maximum principle and the boundary Hopf lemma, we have
\begin{align*}
\Vert w_\epsilon\Vert _{L^\infty(\Omega_\epsilon)}\le \epsilon^2 M \Vert  u_\epsilon \Vert_{L^\infty(\Omega_\epsilon)},\quad \Vert \nabla w_\epsilon\Vert _{L^\infty(\partial \Omega_\epsilon)}\le \epsilon^2 M \Vert  \nabla u_\epsilon \Vert_{L^\infty(\partial \Omega_\epsilon)}.
\end{align*}
From Lemma~\ref{lma21Estimate}, we thus have
\begin{align*}
\Vert w_\epsilon\Vert _{L^\infty(\Omega_\epsilon)}=O(\epsilon^4),\quad \Vert \nabla w_\epsilon\Vert _{L^\infty(\partial \Omega_\epsilon)}=O(\epsilon^3).
\end{align*}
Therefore,
\begin{align*}
|\nabla u_\epsilon(x, f_k(x))|^2=&|\nabla \phi_\epsilon(x,f_k(x))|^2+2\nabla \phi_\epsilon(x, f_k(x))\cdot \nabla w_\epsilon(x,f_k(x))+|\nabla w_\epsilon(x,f_k(x))|^2\\
=&\frac{1}{4}(f_{1}(x) - f_{2}(x))^{2}\epsilon^{2} + O(\epsilon^{4}), \quad k=1,2.
\end{align*}
Thus, the maximum of $|\nabla u_{\epsilon}|^{2}$ along $\partial\Omega_{\epsilon}$ is governed at order $\epsilon^{2}$ by the maximizers of $f_{2}-f_{1}$, and hence the $x$-coordinates of fail points converge to the set $I$ as $\epsilon\to 0^{+}$.
\end{proof}

To prove Theorem \ref{thm13Location}, we need  a higher order expansion of $\nabla u_\epsilon$ in terms of $\epsilon$, as seen in the following proposition.

\begin{proposition}\label{pro22Expansion}
Let $f_{1} \in C^{4}([a, b])$ be convex and $f_{2} \in C^{4}([a, b])$ be concave satisfying \eqref{eq202thinf}. Then for $k = 1, 2$, 
\begin{align} \label{eq204Expansion}%
|\nabla u_{\epsilon}(x, \epsilon f_{k}(x))|^{2} = \frac{1}{4}\left(f_{2}(x) - f_{1}(x)\right)^{2}\epsilon^{2} + \epsilon^{4}\lambda_{k, 2}(x) + O(\epsilon^{6}), 
\end{align}
where $\lambda_{k, 2}$ is a quantity depending on the derivatives of $f_{1}$ and $f_{2}$ up to second order, and is given explicitly by
\begin{equation}\label{lambdak2}
\lambda_{k, 2} = (a_{1}'f_{k} + a_{2}')^{2} + (f_{k} - a_{1})(f_{k}^{2}a_{1}'' + 2f_{k}a_{2}'' - 2a_{3}), 
\end{equation}
for $k = 1, 2$, with 
\begin{equation} \label{ais}
\begin{cases}
a_{1}(x) = \tfrac{1}{2}(f_{1}(x) + f_{2}(x)), 
\\
a_{2}(x) = - \tfrac{1}{2}f_{1}(x)f_{2}(x), 
\\
a_{3}(x) = \tfrac{1}{6}a_{1}''(x)(f_{1}^{2}(x) + f_{1}(x)f_{2}(x) + f_{2}^{2}(x)) + \tfrac{1}{2}a_{2}''(x)(f_{1}(x) + f_{2}(x)).
\end{cases}
\end{equation}
\end{proposition}

\begin{proof}
Let $\Tilde{u}(\epsilon, x, y) = u_{\epsilon}(x, \epsilon y)$. Then $\Tilde{u}$ satisfies
\begin{align} \label{eq:v}
\begin{cases}
\epsilon^2(\Tilde{u})_{xx} + (\Tilde{u})_{yy} = - \epsilon^2, \quad &\text{in $\Omega_{1}$}\\
\Tilde{u} = 0, &\text{on $\partial \Omega_{1}$}.
\end{cases}
\end{align}
We write 
\begin{align}
\label{zhankai}
\Tilde{u}(\epsilon, x, y) = \sum_{k = 0}^{4} \Tilde{u}_k(x, y)\epsilon^k + \Tilde{R}_{4}(\epsilon, x, y), 
\end{align}
where $\Tilde{u}_{0} = \Tilde{u}_{1} = \Tilde{u}_{3} = 0$, 
\begin{equation*}
\Tilde{u}_{2}(x, y) = - \frac{1}{2}(y - f_{1}(x))(y - f_{2}(x)), 
\end{equation*}
and 
\begin{equation*}
\Tilde{u}_{4}(x, y) = - \frac{1}{6}y^{3}a_{1}''(x) - \frac{1}{2}y^{2}a_{2}''(x) + a_{3}(x)y + a_{4}(x), 
\end{equation*}
where the coefficients $a_{1}$, $a_{2}$, $a_{3}$ are given in~\eqref{ais} and 
        \[a_{4}(x) =\frac{1}{6}f_{1}^{3}(x)a_{1}''(x) + \tfrac{1}{2}a_{2}''(x)f_{1}^{2}(x) - a_{3}(x)f_{1}(x).\]
The functions $\Tilde{u}_k$, $1 \leq k \leq 4$ are chosen so that~\eqref{eq:v} is satisfied formally upon matching coefficients of \(\epsilon^{k}\). 

By \eqref{eq:v} and \eqref{zhankai}, the remainder $\Tilde{R}_{4}$ satisfies the equation
\begin{align}
\begin{cases}
\epsilon^2(\Tilde{R}_{4})_{xx} + (\Tilde{R}_{4})_{yy} = - \epsilon^{6}(\Tilde{u}_{4})_{xx}, \quad &\text{in $\Omega_{1}$}\\
\Tilde{R}_{4} = 0, &\text{on $\partial \Omega_{1}$}.
\end{cases}
\end{align}
Hence
\begin{align}\label{expressionofuepsilon}
u_{\epsilon}(x, y) = \Tilde{u}_{2}(x, \epsilon^{ - 1}y)\epsilon^2 + \Tilde{u}_{4}(x, \epsilon^{ - 1}y)\epsilon^{4} + R_{4}(\epsilon, x, y), 
\end{align}
where $R_{4}(\epsilon, x, y) = \Tilde{R}_{4}(\epsilon, x, \epsilon^{ - 1}y)$ is the error term satisfying 
\begin{align}
\begin{cases}
\Delta R_{4}(\epsilon, x, y) = - \epsilon^{4}(\Tilde{u}_{4})_{xx}(x, \epsilon^{ - 1}y) = O(\epsilon^{4}), \quad &\text{in $\Omega_{\epsilon}$}\\
R_{4} = 0, &\text{on $\partial \Omega_{\epsilon}$}.
\end{cases}
\end{align}
There exists a constant $M > 0$ such that $\epsilon^{4}Mu_{\epsilon} \pm R_{4}$ is superharmonic in $\Omega_{\epsilon}$. By the maximum principle, we deduce that 
\begin{align*}
\|R_{4}\|_{L^{\infty}(\Omega_{\epsilon})} \leq \epsilon^{4} M \|u_{\epsilon}\|_{L^{\infty}(\Omega_{\epsilon})}, 
\end{align*}
and consequently, 
\begin{align*}
\|\nabla R_{4}\|_{L^{\infty}(\partial\Omega_{\epsilon})} \leq \epsilon^{4} M \|\nabla u_{\epsilon}\|_{L^{\infty}(\partial \Omega_{\epsilon})}.
\end{align*}
By Lemma \ref{lma21Estimate}, we therefore have
\begin{equation*}
\|R_{4}\|_{L^{\infty}(\Omega_{\epsilon)}} = O(\epsilon^{6}), \quad \|\nabla R_{4}\|_{L^{\infty}(\partial \Omega_{\epsilon})} = O(\epsilon^{5}). 
\end{equation*}
In view of the expression of $u_{\epsilon}$ given by \eqref{expressionofuepsilon} and after somewhat lengthy but straightforward calculations, we have
\begin{equation*}
|\nabla u_{\epsilon}(x, \epsilon f_{k}(x))|^{2} = \epsilon^{2}\lambda_{k, 1}(x) + \epsilon^{4}\lambda_{k, 2}(x) + O(\epsilon^{6}), \quad k = 1, 2, 
\end{equation*}
where 
\begin{align} \label{lambdak} 
\begin{cases}
\lambda_{k, 1} = (f_{2} - f_{1})^2/4\\
\lambda_{k, 2} = (a_{1}'f_{k} + a_{2}')^{2} + (f_{k} - a_{1})(f_{k}^{2}a_{1}'' + 2f_{k}a_{2}'' - 2a_{3})
\end{cases}
\quad k = 1, 2.
\end{align}
Plugging the formulas for $a_{i}$, $i = 1, 2, 3$ in~\eqref{ais} into~\eqref{lambdak} completes the proof.
\end{proof}

Now we are ready to prove Theorem \ref{thm13Location}.

\begin{proof}[Proof of Theorem \ref{thm13Location}]
Since $h: = f_{2} - f_{1}$ attains its maximum at $x = z_{0}$, $f_{1}'(z_{0}) = f_{2}'(z_{0})$. Using this fact, we have 
\begin{equation*}
(a_{1}'f_{1} + a_{2}')^{2}(z_{0}) = (a_{1}'f_{2} + a_{2}')^{2}(z_{0}). 
\end{equation*} 
We let $\Lambda(z_{0}) = |\nabla u_{\epsilon}(z_{0}, \epsilon f_{2}(z_{0}))|^{2} - |\nabla u_{\epsilon}(z_{0}, \epsilon f_{1}(z_{0}))|^{2}$. From Proposition~\ref{pro22Expansion} (equivalently \eqref{eq204Expansion}, direct computation yield
\begin{align*}
\Lambda& = \epsilon^{4}\{(f_{2} - a_{1})(f_{2}^{2}a_{1}'' + 2f_{2}a_{2}'' - 2a_{3}) - (f_{1} - a_{1})(f_{1}^{2}a_{1}'' + 2f_{1}a_{2}'' - 2a_{3})\} + O(\epsilon^{6})\\
& = \frac{1}{6}\epsilon^{4}a_{1}''(f_{2} - f_{1})^{3} + O(\epsilon^{6})\\
& = \frac{1}{12}\epsilon^{4}(f_{1}'' + f_{2}'')(f_{2} - f_{1})^{3} + O(\epsilon^{6}).
\end{align*}
Therefore, $|\nabla u_{\epsilon}(z_{0}, \epsilon f_{2}(z_{0}))| < |\nabla u_{\epsilon}(z_{0}, \epsilon f_{1}(z_{0}))|$ provided that $f_{1}''(z_{0}) + f_{2}''(z_{0}) < 0$, which is equivalent to
$k(z_{0}, f_{2}(z_{0})) > k(z_{0}, f_{1}(z_{0}))$ where $k$ is the curvature function on $\partial \Omega_{1}$.
\end{proof}

Now we consider the case of $C^2$ domains, that is, the derivatives $f_{k}', \, i=1,2$ diverge at the endpoints $a$ and $b$. We have the following result.
\begin{theorem} \label{thm22NarrowBlow}
For each \(\epsilon > 0\), let
\(
\Omega_{\epsilon}: = \{(x, y) \in \mathbb{R}^2: \epsilon f_{1}(x) < y < \epsilon f_{2}(x), \ x\in(a, b)\}
\)
be a bounded \(C^{2}\) convex domain, where \( - f_{1}\) and \(f_{2}\) are positive functions on \((a, b)\) and satisfy \eqref{eq202thinf} at \(x=a, b\). Let \(u_{\epsilon}\) denote the torsion function on \(\Omega_{\epsilon}\). 
Then for any $\delta > 0$, there exists $\epsilon_{0} > 0$ such that for all $\epsilon < \epsilon_{0}$, 
\begin{align}
|\nabla u_{\epsilon}(x, \epsilon f_{k}(x))|^{2} = \frac{1}{4}(f_{1}(x) - f_{2}(x))^{2}\epsilon^{2} + O(\epsilon^{3}), \quad \text{for }x \in (a + \delta, b - \delta), 
\end{align}
for $k = 1, 2$.
\end{theorem}

\begin{theorem} \label{thm23EndSide}
Let $(\Omega_{\epsilon})_{\epsilon > 0}$ be a class of bounded $C^2$ domains as in Theorem \ref{thm22NarrowBlow}.
Then, the endpoints $(a, 0), (b, 0)$ cannot be fail points, provided $\epsilon > 0$ is sufficiently small.
\end{theorem}

\begin{proof}[Proof of Theorem \ref{thm22NarrowBlow}] 
Let 
\begin{equation*}
D_{\epsilon}: = \{a + \delta/2 < x < b - \delta/2\}\cap \Omega_{\epsilon}
\end{equation*} 
and $\eta(x): [a,b] \to [0,1]$ be a smooth cut-off function defined by
\begin{align*}
\eta(x) = 
\begin{cases}
1, \quad &x\in[a + \delta, b - \delta]\\
0, &x \leq a + \delta/2 \quad\text{and}\quad x \geq b - \delta/2
\end{cases}
\end{align*}
with $\eta_{x}$, $\eta_{xx}$ uniformly bounded in $[a, b]$.
Let $w_{\epsilon}(x, y): = u_{\epsilon}(x, y) - \phi_\epsilon(x, y)$ where
$\phi_{\epsilon}(x, y)$ is as in the proof of Theorem \ref{thm12narrow}. That is, 
\begin{align*}
\phi_\epsilon(x, y) = - \frac{1}{2}(y - \epsilon f_{1}(x))(y - \epsilon f_{2}(x)).
\end{align*}
We have
\begin{align}\label{cutoffequation}
\begin{cases}
\Delta(w_{\epsilon}\eta) = \Delta w_{\epsilon} \eta + 2(w_{\epsilon})_x \eta_x + w_{\epsilon}\eta_{xx}\quad &\text{in $D_{\epsilon}$}\\
w_{\epsilon}\eta = 0 &\text{on $\partial D_{\epsilon}$}
\end{cases}
\end{align}
Through direct computation, $\Delta w_{\epsilon} = O(\epsilon^2)$ in $D_\epsilon$, $\|\phi_{\epsilon}\|_{L^{\infty}(D_{\epsilon})} = O(\epsilon^{2})$ and $\|\nabla \phi_{\epsilon}\|_{L^{\infty}(D_{\epsilon})} = O(\epsilon)$. 
Now we estimate $\|(w_{\epsilon})_{x}\|_{L^{\infty}(D_{\epsilon})}$ and $\|w_{\epsilon}\|_{L^{\infty}(D_{\epsilon})}$. By Lemma \ref{lma21Estimate}, we have $$\|(w_{\epsilon})_{x}\|_{L^{\infty}(D_{\epsilon})} \leq \|(u_{\epsilon})_{x}\|_{L^{\infty}(D_{\epsilon})} + \|(\phi_{\epsilon})_{x}\|_{L^{\infty}(D_{\epsilon})} = O(\epsilon)$$ and $$\|w_{\epsilon}\|_{L^{\infty}(D_{\epsilon})} \leq \|u_{\epsilon}\|_{L^{\infty}(D_{\epsilon})} + \|\phi_{\epsilon}\|_{L^{\infty}(D_{\epsilon})} = O(\epsilon^{2}).$$ 
Hence from \eqref{cutoffequation} we have
\begin{align}
\begin{cases}
- \Delta (w_{\epsilon}\eta) = O(\epsilon), \quad &\text{in $D_{\epsilon}$}\\
w_{\epsilon}\eta = 0, &\text{on $\partial D_{\epsilon}.$}
\end{cases}
\end{align}
Similar to the derivation of the gradient estimate of $R_{4}$ in the proof of Proposition \ref{pro22Expansion}, we have 
\begin{equation*}
\|\nabla (w_{\epsilon}\eta)\|_{L^{\infty}(\partial D_{\epsilon})} = O(\epsilon^{2}).
\end{equation*}
Hence 
\begin{equation*}
|\nabla w_{\epsilon}| = O(\epsilon^{2}), \quad \text{uniformly on $\partial \Omega_{\epsilon}\cap \{a + \delta < x < b - \delta\}$}.
\end{equation*}
In view that $u_{\epsilon} = \phi_{\epsilon} + w_{\epsilon}$, for $k = 1, 2$ and $x \in (a + \delta, b - \delta)$, we have
\begin{equation*}
|\nabla u_{\epsilon}(x, \epsilon f_{k}(x))|^{2} = |\nabla \phi_{\epsilon}|^2 + O(\epsilon^3) = \frac{1}{4}(f_{1}(x) - f_{2}(x))^{2}\epsilon^{2} + O(\epsilon^{3}).
\end{equation*}
This finishes the proof.
\end{proof}

In order to prove Theorem \ref{thm23EndSide}, we study the asymptotic behavior of norm of the gradient of torsion function on narrow rectangles.

\begin{lemma}\label{lma25Rectangle}
Let $0 < \epsilon < 1$, $R_{\epsilon} = [0, 1]\times[ - \epsilon, \epsilon]$ be a rectangle in $\R^{2}$ and $u_{\epsilon}$ be the torsion function in $R_{\epsilon}$. Then for any point $p$ on the shorter sides of $R_{\epsilon}$, we have 
\begin{equation*}
|\nabla u_{\epsilon}(p)| = O(\epsilon^{2}).
\end{equation*}
\end{lemma}

\begin{proof}
First, by the moving plane method, $|\nabla u_{\epsilon}|(p) \leq |\nabla u_{\epsilon}|(0, 0)$ for any $p$ in the shorter sides of the rectangle $R_{\epsilon}$.

By separation of variable argument, the torsion function in rectangle $R_{\epsilon}$ is 
\begin{equation*}
u_{\epsilon}(x, y) = \frac{1}{2}x(1 - x) - \frac{2}{\pi^{3}}\sum_{n \geq 1}\frac{1 - ( - 1)^{n}}{n^{3}\cosh(n\pi \epsilon)}\sin (n\pi x)\cosh (n\pi y), 
\end{equation*}
and thus we have
\begin{equation*}
f(\epsilon): = |\nabla u_{\epsilon}|(0, 0) = |(u_{\epsilon})_{x}|(0, 0) = \frac{1}{2} - \frac{4}{\pi^{2}}\sum_{k \geq 0}\frac{1}{(2k + 1)^{2}\cosh(n\pi \epsilon)}. 
\end{equation*}
By direct computation, $f(0) = f'(0) = 0$. Hence $|\nabla u_{\epsilon}|(0, 0) = f(\epsilon) = O(\epsilon^{2})$. We complete the proof.
\end{proof}

\begin{proof}[Proof of Theorem \ref{thm23EndSide}]
Let 
\begin{equation*}
H = \max \Big\{\max_{z\in[a, b]}|f_{1}(z)|, \; \max_{z \in [a, b]}|f_{2}(z)|\Big\}.
\end{equation*} 
Then the rectangle $R_{\epsilon}: = [a, b]\times [ - \epsilon H, \epsilon H]$ contains $\Omega_{\epsilon}$ with $(a, 0), (b, 0) \in \partial R_{\epsilon}\cap \partial \Omega_{\epsilon}$. Let $u_{\Omega_{\epsilon}}$ be the torsion function in $\Omega_{\epsilon}$ and $v_{\epsilon}$ be the torsion function in $R_{\epsilon}$. By the maximum principle, Hopf lemma and Lemma \ref{lma25Rectangle}, we have
\begin{align*}
|\nabla u_{\Omega_{\epsilon}}(p)| \leq |\nabla v_{\epsilon}(p)| = O(\epsilon^{2})
\end{align*}
for $p = (a, 0)$ or $(b, 0)$.

On the other hand, the inradius of $\Omega_{\epsilon}$ is bounded below by $c\epsilon$ for some $c > 0$. By the Hopf Lemma and the torsion function's formula on disks, we derive that the norm of the gradient of the torsion function on $\Omega_{\epsilon}$ at the contact point of some maximal disk contained in $\Omega_{\epsilon}$ is bounded below by some $c\epsilon$, with $c > 0$. Consequently, fail points cannot occur at $(a, 0)$ or $(b, 0)$ when $\epsilon > 0$ is small enough.
\end{proof}


\section{Location and uniqueness of fail points on triangles} \label{Sect3Triangle}

In this section, we study the location and number of fail points in triangles. Throughout this paper, every triangle is understood as an open domain (i.e., the interior of the corresponding closed triangle). We write \(\triangle_{ABC}\) for the open triangle with vertices \(A\), \(B\), and \(C\). Given two points \(A\) and \(B\), we use \(AB\) to denote the closed line segment joining them, and \(|AB|\) to denote its length.

\begin{proposition} \label{PropCompareLength}
For a triangle, the fail point must lie exclusively on the longest side.
\end{proposition}

\begin{proof}
Let $u$ be the torsion function in the triangle $\triangle_{ABC}$. Without loss of generality, assume that $|AC| < |AB|$. We claim that for any point $p$ on the shorter side $AC$, there exists a point $p'$ on the longer side $AB$ such that $|\nabla u(p)| < |\nabla u(p')|$. 

By employing suitable rigid transformation of the triangle, we may place vertex $A$ at the origin, $B$ on the positive $x$-axis, and $C$ in the upper half-plane. Let $C'$ be the point on $AB$ such that $|AC'| = |AC|$. Further, let $Q$ be a point on $BC$ such that the line $AQ$ bisects the angle $\angle BAC$. With this construction, the triangles $\triangle_{AQC}$ and $\triangle_{AQC'}$ are symmetric with respect to the line $AQ$ (see Figure \ref{fig31} for illustration). 

\begin{figure}[htp]
\centering
\begin{tikzpicture}[scale = 2]
\pgfmathsetmacro\jiaoA{70}; 
\pgfmathsetmacro\LenB{4.0000}; 
\pgfmathsetmacro\LenC{\LenB*0.4000}; 
\pgfmathsetmacro\Cx{\LenC*cos(\jiaoA)}; 
\pgfmathsetmacro\Cy{\LenC*sin(\jiaoA)}; 
\pgfmathsetmacro\px{0.6*\Cx}; 
\pgfmathsetmacro\py{0.6*\Cy}; 
\pgfmathsetmacro\qx{0.6*\LenC}; 
\pgfmathsetmacro\qy{0}; 
\pgfmathsetmacro\Qx{\LenC*sin(\jiaoA)*\LenB/((\LenB - \LenC*cos(\jiaoA))*(tan(\jiaoA/2) + \LenC*sin(\jiaoA)/(\LenB - \LenC*cos(\jiaoA))))}; 
\pgfmathsetmacro\Qy{tan(\jiaoA/2)*\Qx}; 
\fill[gray, yellow] (0, 0) -- (\Qx, \Qy) -- (\LenC, 0) -- cycle; 
\fill[gray, green] (0, 0) -- (\Qx, \Qy) -- (\Cx, \Cy) -- cycle; 
\draw[black, dashed] (\Qx, \Qy) -- (\LenC, 0); 
\draw[thick, red] (0, 0) -- (\Qx, \Qy); 
\fill (\px, \py) circle (0.02) node[left] {\small ${p}$}; 
\fill (\qx, \qy) circle (0.02) node[below] {\small ${p'}$}; 
\draw[dotted, thick] (\px, \py) -- (\qx, \qy); 
\node[below left] at (0, 0) {$A$}; 
\node[below right] at (\LenB, 0) {$B$}; 
\node[left] at (\Cx, \Cy) {$C$}; 
\fill (\LenC, 0) circle (0.02) node[below] {\small ${C'}$}; 
\fill (\Qx, \Qy) circle (0.02) node[above right] {\small ${Q}$}; 
\draw[thick] (0, 0) -- (\LenB, 0) -- (\Cx, \Cy) -- cycle; 
\end{tikzpicture}
\caption{An illustration of the proof of Proposition \ref{PropCompareLength}: If $|AB| > |AC|$ and $AQ$ bisects the angle $\angle BAC$, then $|\nabla u(p)| < |\nabla u(p')|$, for any $p \in AC$, with $p'$ being the reflection of $p$ about $AQ$.}
\label{fig31}
\end{figure}
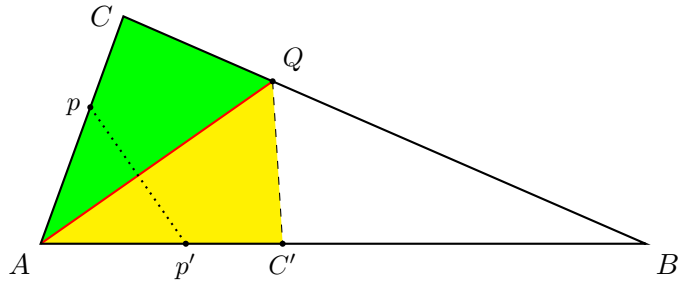

For any $(x, y) \in \triangle_{AQC}$, define 
\begin{equation*}
{w} (x, y) = {u} (x, y ) - {u} (x', y'), 
\end{equation*}
where $(x', y')$ is the reflection of $(x, y)$ with respect to the angle bisector $AQ$. That is, 
\begin{align*}
(x', y') = (x\cos \alpha + y\sin\alpha, x\sin\alpha - y\cos\alpha), 
\end{align*}
where $\alpha$ is the interior angle at vertex $A$. Observing that 
\begin{equation*}
\begin{cases}
\Delta w = 0 & \text{in } \triangle_{AQC}, \\
w = 0 & \text{on } AC \cup AQ, \\
w < 0 & \text{on } CQ \setminus \{C, Q\}, 
\end{cases}
\end{equation*} 
so by the strong maximum principle, 
\begin{equation} \label{eq302reflection}
w < 0 \text{ in } \triangle_{AQC}.
\end{equation}
Since $w$ vanishes on $AC$, the Hopf lemma yields that for any $p = (r\cos\alpha, r\sin\alpha) \in AC \setminus \{A, C\}$, 
\begin{equation} 
\partial_{\nu_p}w(p) = \frac{\partial u}{\partial \nu_{p}}(p) - \frac{\partial u}{\partial \nu_{p'}}(p') = - |\nabla u|(p) + |\nabla u|(p') > 0, 
\end{equation} 
where $p' = (r, 0) \in AB$ is the reflection of $p$ about the angle bisector $AQ$. 

Therefore, for any point $p$ on $AC$, there exists a point $p'$ on $AB$ such that $|\nabla u(p)| < |\nabla u(p')|$. This shows that no fail point can exist on the side $AC$. 
\end{proof}

In the following, when we refer to a \textit{critical point} of $|\nabla u|^2$ on an open side $\Gamma$ of a triangle, we mean a critical point of the function
\begin{align*}
&\Gamma \to \mathbb{R}, \\
&x \mapsto |\nabla u|^2(x).
\end{align*}

\begin{proposition} \label{PropLocation}
Let \(u\) be the torsion function on \(\Omega = \triangle_{ABC}\). Then, for each open side \(\Gamma\) of \(\Omega\), the restriction of \(|\nabla u|^{2}\) to \(\Gamma\) has at least one critical point, and every such critical point lies on the segment joining the midpoint of \(\Gamma\) to the foot of the altitude from the opposite vertex onto \(\Gamma\).
\end{proposition}

\begin{proof}

It suffices to consider the side \(AB\).
Since \(|\nabla u| = 0\) at \(A\) and \(B\) (by the \(C^{1}\) boundary regularity of the torsion function on triangles) and \(|\nabla u| > 0\) in the interior of \(AB\), the continuous function \(|\nabla u|^{2}\) attains its maximum on the closed segment \(AB\) at some interior point. In particular, \(|\nabla u|^{2}\) has at least one critical point on the open segment \(AB\).
        

Without loss of generality, assume \(|CA|\leq |CB|\), place \(AB\) on the horizontal axis with \(A=(0,0)\) and \(B=(\ell_{0},0)\) for some \(\ell_{0}>0\), and take \(C\) in the upper half-plane. 
        Let \(M=(\ell_{0}/2,0)\) be the midpoint of \(AB\), and let \(F=(\ell_{1},0)\) be the foot of the altitude from \(C\) onto the line \(AB\), where \(\ell_{1}\leq \ell_{0}/2\). Let \(Q=(\ell_{0}/2,h)\in BC\) be the point such that \(QM\perp AB\); see Figure~\ref{fig32}.

\begin{figure}[htp]
\centering
\begin{tikzpicture}[scale = 2]
\pgfmathsetmacro\jiaoA{70}; 
\pgfmathsetmacro\LenB{4.0000}; 
\pgfmathsetmacro\LenC{\LenB*0.4000}; 
\pgfmathsetmacro\Cx{\LenC*cos(\jiaoA)}; 
\pgfmathsetmacro\Cy{\LenC*sin(\jiaoA)}; 
\pgfmathsetmacro\Mx{\LenB/2}; 
\pgfmathsetmacro\My{0}; 
\pgfmathsetmacro\Qx{\Mx}; 
\pgfmathsetmacro\Qy{\LenB*\LenC*sin(\jiaoA)/(2*(\LenB - \LenC*cos(\jiaoA)))}; 
\pgfmathsetmacro\Fx{\Cx}; 
\pgfmathsetmacro\Fy{0}; 
\node[below left] at (0, 0) {$A$}; 
\node[below right] at (\LenB, 0) {$B$}; 
\node[left] at (\Cx, \Cy) {$C$}; 
\draw[thick] (0, 0) -- (\LenB, 0) -- (\Cx, \Cy) -- cycle; 
\fill[gray, yellow] (0, 0) -- (\Qx, \Qy) -- (\Mx, \My) -- cycle; 
\fill[gray, green] (\LenB, 0) -- (\Qx, \Qy) -- (\Mx, \My) -- cycle; 
\draw[black, dashed] (0, 0) -- (\Qx, \Qy); 
\draw[dotted, thick] (\Fx, \Fy) -- (\Cx, \Cy); 
\draw[thick, red] (\Qx, \Qy) -- (\Mx, \My); 
\fill (\Fx, \Fy) circle (0.02) node[below] {\small ${F}$}; 
\fill (\Qx, \Qy) circle (0.02) node[above right] {\small ${Q}$}; 
\fill (\Mx, \My) circle (0.02) node[below] {\small ${M}$}; 
\end{tikzpicture}
\caption{An illustration of proof of Proposition \ref{PropLocation}: $M$ is the midpoint of $AB$, $F$ is the foot of the altitude on $AB$, then restricted to the side $AB$, the maximal gradient of the torsion function occurs strictly between $F$ and $M$.}
\label{fig32}
\end{figure}
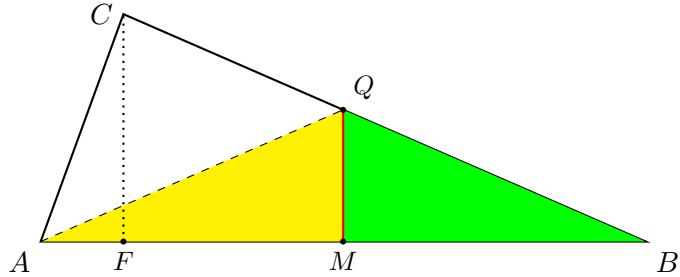

Using the same reflection argument as in the proof of Proposition~\ref{PropCompareLength}, one obtains
\begin{equation*}
{u}(x, y) - {u}(2\lambda - {x}, {y}) < 0 \quad\text{for } {x} > \lambda > \ell_{0}/2, 
\end{equation*}
and 
\begin{equation} \label{eq305a}
u_{x} < 0 \text{ in } \Delta_{MBQ}.
\end{equation}
Furthermore, we have $u_{x} \leq 0$ on $QM$, with equality if and only if $F = M$, i.e., $|CA| = |CB|$. Since 
\begin{equation*}
\begin{cases}
\Delta u_{x} = 0 & \text{in } \triangle_{MBQ}, \\
u_{x} = 0 & \text{on } MB, \\
u_{x} < 0 & \text{in the interior of } \triangle_{MBQ}, 
\end{cases}
\end{equation*}
the Hopf lemma yields 
\begin{equation*}
{u}_{xy}(p) < 0 \quad \text{for all } p \in MB \setminus \{B, M\}.
\end{equation*}
On \(AB\) we have \(u_{x} = 0\), so \(|\nabla u|^{2} = u_{y} ^{2}\) along \(AB\), and thus \(|\nabla u|^{2}\) has no critical point on $MB \setminus \{B, M\}$. Consequently, any critical point of $|\nabla u|^2$ on $AB$ must lie on the segment $AM$. 


Next, using a similar reflection argument with respect to the line \(CF\), and noting that the reflection of \(\triangle_{AFC}\) across \(CF\) is still contained in \(\triangle_{ABC}\), we obtain 
\begin{equation} \label{eq305b}
u_{x} < 0 \text{ in } \Delta_{CAF}, 
\end{equation}
provided the angle at \(A\) is acute. Moreover, 
\begin{equation} \label{eq305c}
u_{x} = 0 \text{ on } {CF} \quad \text{if and only if} \quad |CA| = |CB|.
\end{equation}
Assume now that \(A\) is acute. Then \eqref{eq305b} and the Hopf lemma give 
\begin{equation*}
{u}_{xy}(p) > 0 \quad \text{for all } p \in AF \setminus \{A, F\}. 
\end{equation*}
Therefore, $|\nabla u|^2$ has no critical point on $AF \setminus \{A, F\}$, which implies that every critical point of $|\nabla u|^2$ on $AB$ must lie on the segment $FM$. 
This completes the proof.
\end{proof}

\begin{proof}[Proof of Theorem \ref{thm14LongestSide}]
Observe that any fail point in the triangle $\triangle_{ABC}$ must be a critical point of the squared norm of the gradient of the torsion function on some side. By combining Proposition \ref{PropCompareLength} and Proposition \ref{PropLocation}, the theorem follows immediately.
\end{proof}

\begin{figure}[htp]
\centering
\begin{tikzpicture}[scale = 2]
\pgfmathsetmacro\jiaoA{70}; 
\pgfmathsetmacro\LenB{4.0000}; 
\pgfmathsetmacro\LenC{\LenB*0.4000}; 
\pgfmathsetmacro\Cx{\LenC*cos(\jiaoA)}; 
\pgfmathsetmacro\Cy{\LenC*sin(\jiaoA)}; 
\pgfmathsetmacro\lam{0.5}; 
\pgfmathsetmacro\ex{ - 2*\lam + \LenC*sin(\jiaoA) - tan(\jiaoA/2)*\LenC*cos(\jiaoA)}; 
\pgfmathsetmacro\ey{\LenC*sin(\jiaoA)*tan(\jiaoA/2) + \LenC*cos(\jiaoA)}; 
\pgfmathsetmacro\ez{tan(\jiaoA/2)*tan(\jiaoA/2) + 1}; 
\pgfmathsetmacro\Ex{(\ex*tan(\jiaoA/2) + \ey)/\ez}; 
\pgfmathsetmacro\Ey{( - \ex + tan(\jiaoA/2)*\ey)/\ez}; 
\pgfmathsetmacro\Px{\lam/(tan(\jiaoA) - tan(\jiaoA/2))}; 
\pgfmathsetmacro\Py{tan(\jiaoA)*\Px}; 
\pgfmathsetmacro\Qx{(\LenC*sin(\jiaoA)*\LenB/(\LenB - \LenC*cos(\jiaoA)) - \lam)/(tan(\jiaoA/2) + \LenC*sin(\jiaoA)/(\LenB - \LenC*cos(\jiaoA)))}; 
\pgfmathsetmacro\Qy{tan(\jiaoA/2)*\Qx + \lam}; 
\node[above = 4pt, left] at (0, 0) {$A$}; 
\node[above = 4pt, right] at (\LenB, 0) {$B$}; 
\node[below = 2pt, left] at (\Cx, \Cy) {$C$}; 
\draw[thick] (0, 0) -- (\LenB, 0) -- (\Cx, \Cy) -- cycle; 
\fill (\Ex, \Ey) circle (0.02) node[below right] {\small ${C_{\lambda}}$}; 
\fill (\Px, \Py) circle (0.02) node[above left] {}; 
\fill (\Qx, \Qy) circle (0.02) node[below] {}; 
\fill[gray, green, draw = black] (\Px, \Py) -- (\Cx, \Cy) -- (\Qx, \Qy) -- cycle; 
\fill[gray, yellow, draw = black, dashed] (\Px, \Py) -- (\Ex, \Ey) -- (\Qx, \Qy) -- cycle; 
\draw[red, thick] ({\Px - 0.3*(\Qx - \Px)}, {\Py - 0.3*(\Qy - \Py)}) -- ({\Px + 1.28*(\Qx - \Px)}, {\Py + 1.28*(\Qy - \Py)}) node[right] {\normalsize $l: \; y = \lambda + x\tan\tfrac{\theta}{2}$}; 
\end{tikzpicture}
\caption{The moving plane method for the semilinear case}
\label{fig33MMP}
\end{figure}
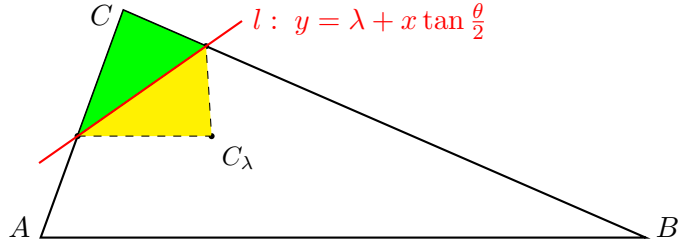

The two results above also apply to positive solutions of the semilinear equation
\begin{equation}\label{eq307sl}
\begin{cases}
- \Delta u = f(u) & \text{in } \Omega, \\
u = 0 & \text{on } \partial \Omega, 
\end{cases}
\end{equation}
where $f$ is a locally Lipschitz function. 

\begin{remark}
Let \(u\) be a positive classical solution of \eqref{eq307sl} in a planar triangle \(\Omega\). Then the maximum of \(|\nabla u|^{2}\) on \(\partial\Omega\) is attained only on the longest side of the triangle. Moreover, for any side \(\Gamma\) of \(\Omega\), let \(\tau\) denote the unit tangential direction along \(\Gamma\). Then the set 
\(
\{x \in \Gamma: |\nabla u(x)| > 0 \ \text{and}\ \partial_{\tau}|\nabla u|(x) = 0\}
\)
is nonempty and is contained in the segment joining the midpoint of \(\Gamma\) to the foot of the altitude from the opposite vertex onto \(\Gamma\). 

The proof follows standard moving-plane arguments \cite{GNN79}, rather than a direct reflection, and is therefore omitted. Figure~\ref{fig33MMP} illustrates the idea: by continuously translating the red line \(l\) from left to right until it passes through the vertex \(A\), one obtains \eqref{eq302reflection}.
\end{remark}

Next, we prove Theorem \ref{thm15Uniqueness}.

\begin{proof}[Proof of Theorem \ref{thm15Uniqueness}]
We restrict our attention to the critical points of \(\lvert \nabla u \rvert^{2}\) along the open side \(AB\), noting that the analysis on the remaining sides \(AC\) and \(BC\) can be carried out in a similar manner. After a suitable rigid motion, we may assume that \(AB\) lies on the \(x\)-axis, with \(A\) to the left of \(B\), and that \(C\) lies in the upper half-plane. 
Note that there exists a critical point of $|\nabla {u}|^{2}$ on ${A}{B}$. An interior point ${P}$ on the side ${A}{B}$ is a critical point of $|\nabla {u}|^{2}$ if and only if ${u}_{xy}({P}) = 0$. 
We denote the nodal line of the partial derivative \( {u}_{x} \) by 
\begin{equation*}
\mathcal{Z}({u}_{x}): = \overline{\{ (x, y) \in \triangle_{{A}{B}{C}}: {u}_{x}(x, y) = 0 \}}. 
\end{equation*}

\textbf{Step 1.} We prove that the nodal line \(\mathcal{Z}({u}_{x})\) is a simple \(C^{1}\) arc connecting the vertex \({C}\) to a point \( {P} \) on the side \({A}{B}\). Since \( {u}_{x} \) is harmonic and by the structure theory for nodal sets of planar harmonic functions, the set \(\mathcal{Z}({u}_{x})\) cannot terminate in the interior of \(\triangle_{ABC}\) nor form a closed loop. Hence, \(\mathcal{Z}({u}_{x})\) is a finite union of properly immersed \(C^{1}\) arcs. Moreover, by the Hopf Lemma, 
\begin{equation*}
{u}_{x} \neq 0 \quad \text{on the interiors of the sides } {A}{C} \text{ and } {B}{C}.
\end{equation*}
Thus, each immersed \(C^{1}\) arc in \(\mathcal{Z}({u}_{x})\) must have two endpoints, one of which is \({C}\) and the other lies on \({A}{B}\). If, by way of contradiction, there were two such arcs, then, since \( {u}_{x} = 0 \) on \({A}{B}\), there would exist a nodal domain \(\mathcal{D}\) of \( {u}_{x} \) with 
\begin{equation*}
\partial \mathcal{D} \subset \mathcal{Z}({u}_{x}) \cup {A}{B}.
\end{equation*}
The strong maximum principle would then force \( {u}_{x} \equiv 0 \) in \(\mathcal{D}\), which is impossible. This concludes Step 1.

\textbf{Step 2}. We show that for a point ${P}$ on the interior of side ${A}{B}$, there are three cases: 
\begin{enumerate}
\item 
If \( {u}_{xy}({P}) \neq 0 \), then no nodal line of \( {u}_{x} \) emanates from \(P\).
\item 
If \( {u}_{xy}({P}) = 0 \) but \( {u}_{xxy}({P}) \neq 0 \), then there is exactly one nodal line of \( {u}_{x} \) emanating from \(P\), and the tangential line of this nodal line at \( {P} \) is perpendicular to the side \({A}{B}\).
\item 
If \( {u}_{xy}({P}) = 0 \) and \( {u}_{xxy}({P}) = 0 \), then there are at least two nodal lines of \( {u}_{x} \) emanating from \(P\).
\end{enumerate}
To facilitate this analysis, we define the function \({v}\) as the odd extension of \({u}_{x}\) with respect to the side \({A}{B}\); namely, \({v}\) is defined on the double domain \(\tilde{\Omega}\) by
\begin{equation*}
{v}(x, y) = 
\begin{cases}
{u}_{x}(x, y) & \text{ when } y > 0, \\
0 & \text{ when } y = 0, \\
- {u}_{x}(x, - y) & \text{ when } y < 0.
\end{cases}
\end{equation*}
Then \({v}\) is harmonic in \(\tilde{\Omega}\) and is not identically zero. If \( {u}_{xy}({P}) \neq 0 \) (that is, if \( {v}_{y}({P}) \neq 0 \)), then the implicit function theorem ensures that, in a neighborhood of \(P\), the nodal set of \( {v} \) coincides with the line \({A}{B}\). On the other hand, if \( {u}_{xy}({P}) = 0 \) (which implies \( {v}_{y}({P}) = 0 \)) and noting that \( {v}_{x}(P) = {u}_{xx}(P) = 0 \), we have
\begin{equation*}
{v}(P) = 0, \quad {v}_{x}(P) = {v}_{y}(P) = 0.
\end{equation*}
Since \( {v} \not\equiv 0 \), by \cite[Proposition 4.1]{HHH99} the point \(P\) is a zero of \( {v} \) of order \( {k} \) for some integer \( k \geq 2 \). Moreover, the zero set \(\{v = 0\}\) consists of exactly \(2{k}\) curves emanating from \(P\) in \(\tilde{\Omega}\), and the tangents to these branch curves partition the full circle into \(2{k}\) equal angles. In view of the odd symmetry of \( {v} \) about \({A}{B}\), there are exactly \( {k} - 1 \) of these nodal lines correspond to distinct branches of \( {u}_{x} \) that are contained in the interior of \(\triangle_{ABC}\). Furthermore, from the equation satisfied by \( {u} \), one may deduce that
\begin{equation*}
{v}_{xx}(P) = 0 \quad \text{and} \quad {v}_{yy}(P) = \bigl({u}_{yy}\bigr)_{x} (P) = \bigl( - 1 - {u}_{xx}\bigr)_{x} (P) = - {u}_{xxx}(P) = 0.
\end{equation*}
The case \( {k} = 2 \) occurs if and only if the hessian \( {D}^{2}{v}(P) \) has a nonzero determinant, that is, \( {v}_{xy}(P) \neq 0 \). This completes Step 2.

Observe that there exists a critical point \( {P} \) of \(|\nabla {u}|^{2}\) on the interior of \({A}{B}\). By Step 2, there exists a nodal line of \({u}_{x}\) emanating from \({P}\). From Step 1, the endpoints of the simple nodal curve \(\mathcal{Z}({u}_{x})\) must be \({C}\) and \({P}\). Combining this with Step 2, we conclude that \({P}\) must be the unique critical point of \(|\nabla {u}|^{2}\) on the interior of \({A}{B}\), and there is exactly one nodal line of \({u}_{x}\) emanating from \({P}\). 
In summary, 
\begin{equation}\label{eq308}
\begin{aligned}
&u_{xy}(P) = 0 \quad\text{and}\quad u_{xxy}(P)\neq 0, \\
&\text{the nodal set \(\mathcal{Z}(u_{x} )\) is an analytic curve connecting \(P\) and \(C\), and}\\
&\text{the tangent line to \(\mathcal{Z}(u_{x} )\) at \(P\) is perpendicular to the side \(AB\).}
\end{aligned}
\end{equation}
This implies that the critical point \(P\) of \(|\nabla u|^{2}\) on \(AB\) is nondegenerate.
This completes the proof. 
\end{proof}

By Theorem \ref{thm14LongestSide} and Theorem \ref{thm15Uniqueness}, we arrive at the following corollary: 
\begin{corollary}
For any triangle, the number of fail points equals the number of longest sides. Consequently, a triangle is equilateral if and only if it has exactly three fail points.
\end{corollary}

Last, we remark that through the study of finding the location of fail points on triangles, we obtain another rigidity result for isosceles triangles as a byproduct. 

%

\begin{corollary} 
Let \(\Omega\) be an open triangle, and let \(u\) denote the torsion function in \(\Omega\).
Then \(\Omega\) is isosceles if and only if there exists a unit vector \(\gamma\) such that the nodal line
\begin{equation*}
\mathcal{Z}(\partial_{\gamma}u) = \overline{\{x \in \Omega: \partial_{\gamma}u(x) = 0\}}
\end{equation*}
contains a straight line segment of positive length.
\end{corollary}

\begin{proof}
The ``only if'' direction follows from \eqref{eq305c}. We now prove the ``if'' direction.

Assume that there exist a unit vector \(\gamma\) and a line \(L\) such that
\(\mathcal{Z}(\partial_{\gamma}u)\) contains a nontrivial line segment in \(L \cap \Omega\). Since \(\partial_{\gamma}u\) is harmonic, hence real-analytic in \(\Omega\), 
it follows that \(\partial_{\gamma}u\equiv 0\) on \(\overline{L_{0}}\) where \({L}_{0} = {L} \cap \Omega\). 

If an endpoint of \(L_{0}\) lies in the interior of a side of \(\Omega\), then the Hopf lemma implies
that the inward normal derivative of \(u\) at that point is strictly positive. Since
\(\mathcal{Z}(\partial_{\gamma}u)\supset L_{0}\), this forces \(\gamma\) to be tangent (hence parallel) to
the side whose interior contains that endpoint.

Therefore, after relabeling if necessary, we may assume that one endpoint \(P\) of \(L_{0}\) lies in the
interior of the side \(AB\), while the other endpoint is the opposite vertex \(C\). Place \(AB\) on the
horizontal axis and choose coordinates so that \(\gamma\) is horizontal. By \eqref{eq308}, the nodal line
\(\mathcal{Z}(u_{x} ) = \mathcal{Z}(\partial_{\gamma}u)\) has tangent at \(P\) perpendicular to \(AB\); hence the
segment \(\overline{L_{0}} = CP\) is perpendicular to \(AB\). In particular, \(CP\perp AB\) and
\(\partial_{\gamma}u = 0\) on \(CP\). Applying \eqref{eq305c} yields \(|CA| = |CB|\). Thus \(\triangle_{ABC}\) is
isosceles with base \(AB\).
\end{proof}

\section{Fail points on nearly equilateral triangles}\label{Sect4EquTri}

In this section, we further analyze the gradient norm of the torsion function on nearly equilateral triangles. Our goal is to answer Questions \ref{OpenQue16}-\ref{OpenQue17} for some nearly equilateral triangles.

%
We first need to prove Theorem~\ref{thm18ctm}. To do so, we let \(u(t; x, y)\) be the torsion function in \(\Omega_{t} = \triangle_{A_{t}B_{t}C_{t}}\), where \(A_{t} = ( - \sqrt{3}/3, 0)\), \(B_{t} = (\sqrt{3}/3 + t, 0)\), and \(C_{t} = (0, 1)\). 
Consider the smooth mapping
\begin{align} \label{eqCtm02}
F_{t}(x, y) = \bigg( \Big(1 + \frac{\sqrt{3}}{2} t \Big)x - \frac{y - 1}{2} t, \, y \bigg).
\end{align}
The map \(F_{t}\) is linear and fixes \(A_{0}\) and \(C_{0}\) while sending \(B_{0}\) to \(B_{t}\). In particular, \(\Omega_{t} = F_{t}(\Omega_{0})\). A direct computation shows that the pullback \(v(t; x, y): = u(t; F_{t}(x, y))\) satisfies
\begin{align}
\begin{cases}
\left(1 + \frac{t^{2}}{4}\right)\left(1 + \frac{\sqrt{3}}{2}t\right)^{ - 2}v_{xx} + t\left(1 + \frac{\sqrt{3}}{2}t\right)^{ - 1}v_{xy} + v_{yy} = - 1 & \text{in } \Omega_{0}, \\
v = 0 & \text{on } \partial\Omega_{0}. 
\end{cases}
\end{align}
By standard elliptic regularity and the implicit function theorem, the map $t \mapsto v(t; \cdot)$ is smooth from $[0, 1]$ to $C^{1}(\overline{\Omega_{0}})$; in particular, $v(t; x, y)$ depends smoothly on $t$. Assume that $v(t; x, y) = v_{0}(x, y) + tv_{1}(x, y) + t^2 v_{2}(x, y) + O(t^3)$ as $t \to 0$. Then $v_{0}$ and ${v}_{1}$ satisfy
\begin{align}
\begin{cases}
- \Delta v_{0} = 1 & \text{in } \Omega_{0}, \\
v_{0} = 0 & \text{on } \partial\Omega_{0}, 
\end{cases}
\end{align}
and 
\begin{align}\label{eqCtm06}
\begin{cases}
- \Delta v_{1} = (v_{0})_{xy} - \sqrt{3}(v_{0})_{xx} & \text{in } \Omega_{0}, \\
v_{1} = 0 & \text{on } \partial\Omega_{0}. 
\end{cases}
\end{align}
One can deduce that 
\begin{equation}\label{eqCtm07v0}
v_{0}(x, y) = \frac{1}{4}y(y - 1 + \sqrt{3} x)(y - 1 - \sqrt{3} x) = \frac{1}{4}y^{3} - \frac{1}{2}y^{2} + \frac{1}{4}y - \frac{3}{4}x^{2}y, 
\end{equation}
and 
\begin{equation}\label{eqCtm07v0xx}
(v_{0})_{xx}(x, y) = - \frac{3y}{2}, \quad (v_{0})_{xy}(x, y) = - \frac{3x}{2}.
\end{equation}
Therefore, \eqref{eqCtm06} becomes
\begin{align}\label{eqCtm08v1}
\begin{cases}
- \Delta v_{1} = \frac{3\sqrt{3}}{2}y - \frac{3x}{2} & \text{in } \Omega_{0}, \\
v_{1} = 0 & \text{on } \partial\Omega_{0}.
\end{cases}
\end{align}

In order to prove Theorem \ref{thm18ctm}, we shall first derive the formula for the nodal line of $u_{x} $. Using the notations introduced above, we have the following proposition.

\begin{proposition}\label{prop41Step2}
Let $\Omega_{t}$ be the triangle as in Theorem \ref{thm18ctm} and $u(t; \cdot)$ be the torsion function on $\Omega_{t}$. Then, there exists a constant $\delta_{0} > 0$ and a function $\varphi(t; y)$, such that the nodal line of $u_{x} (t; \cdot)$ is given by $\mathcal{N}_{t} = \{(\varphi(t; y), y): y \in (0, 1)\}$ when $|t| < \delta_{0}$. $\mathcal{N}_{t}$ splits $\Omega_{t}$ into two connected subregions, and the fail point $p_{t}$ is exactly the contact point of $\mathcal{N}_{t}$ with $A_{t}B_{t}$. Moreover, we have the formula
\begin{align}
\varphi(t; y) = t\Big(\frac{2(v_{1})_{x}(0, y)}{3y} - \frac{y}{2} + \frac{1}{2}\Big) + O(t^{2}) \quad \text{when } |t| < \delta_{0}, 
\end{align}
where $v_{1}$ is the solution to \eqref{eqCtm08v1}.
\end{proposition}
\begin{proof}
By the analysis in Section~\ref{Sect3Triangle}, starting from the fail point $p_{t} \in A_{t}B_{t}$, there is a nodal line $\mathcal{N}_{t}$ of $u_{x} $ that enters the interior of the triangle and terminates at the vertex $C_{t}$. By the maximal principle, $\mathcal{N}_{t}$ splits $\Omega_{t}$ into two subregions: on the left of $\mathcal{N}_{t}$, one has $u_{x} > 0$, whereas on the right of $\mathcal{N}_{t}$, one has $u_{x} < 0$.

Since $v(0; x, y)$ is the torsion function in the equilateral triangles $\Omega_{0}$, we have $v_{x}(0; 0, y) = 0, v_{xx}(0; 0, y) = - 3y/2$. Let 
\begin{align*}
V(t; x, y): = 
\begin{cases}
\frac{v_{x}(t; x, y)}{y}, \quad &y > 0\\
v_{xy}(t; x, 0), &y = 0\\
- \frac{v_{x}(t; x, - y)}{y}, &y < 0
\end{cases}
\end{align*}
Then $V(t; x, y)$ is a smooth function defined on the doubled domain obtained by reflecting $\Omega_{0}$ across $A_{0}B_{t}$. Moreover, $V(0; 0, y) = 0$ and $V_{x}(0; 0, y) = - 3/2$. By the implicit function theorem, there exist $\delta_{0} > 0$ and a smooth function $\psi(t; y)$ defined for $|t| < \delta_{0}$ and $|y| < 1$ such that $\psi(0; y) = 0$ and for \(|t| < \delta_{0}\) and \(|y| < 1\) the equation \(V(t; x, y) = 0\) has the unique solution \(x = \psi(t; y)\). Since $v_{x}(t; x, y) = (1 + \sqrt{3}t/2)u_{x}(t; F_{t}(x, y))$, where $F_{t}(\cdot)$ is defined by \eqref{eqCtm02}, there exists a function $\varphi(t; y)$ that depends smoothly on $t$, such that 
\begin{equation*}
\mathcal{N}_{t} = \{(\varphi(t; y), y): y \in [0, 1)\} \quad \text{for } |t| < \delta_{0}.
\end{equation*}
Noting that \(u(0; x, y)\) is symmetric with respect to the \(y\)-axis, we have \(\varphi(0; y) = 0\). Writing
\begin{equation*}
\varphi(t; y) = \varphi_{0}(y) + t \varphi_{1}(y) + O(t^{2}) \quad \text{as } t \to 0,
\end{equation*}
it follows that \(\varphi_{0}(y) = 0\). The inverse map $F_{t}^{ - 1}$ of $F_{t}$ is given by 
\begin{equation}\label{eqCtm11Finv}
F_{t}^{-1}(x, y)
= \Bigl(\frac{2x + t(y-1)}{2 + \sqrt{3}t}, \, y\Bigr).
\end{equation}
Therefore, 
\begin{align*}
F_{t}^{ - 1}\bigl(\varphi(t; y), y\bigr)
& = \bigl(\varphi(t; y), y\bigr)
+ t\Big(\frac{y - 1}{2} - \frac{\sqrt{3}}{2}\, \varphi(t; y), \, 0\Big) + O(t^{2})\\
& = \bigl(\varphi_{0}(y), y\bigr)
+ t\Big(\varphi_{1}(y) + \frac{y - 1}{2} - \frac{\sqrt{3}}{2}\, \varphi_{0}(y), \, 0\Big) + O(t^{2}), 
\end{align*}
and 
\begin{align*}
0 & = v_{x}\Bigl(t; F_{t}^{ - 1}\bigl(\varphi(t; y), y\bigr)\Bigr) \\ 
& = (v_{0})_{x}\Bigl(F_{t}^{ - 1}\bigl(\varphi(t; y), y\bigr)\Bigr)
+ t\, (v_{1})_{x}\Bigl(F_{t}^{ - 1}\bigl(\varphi(t; y), y\bigr)\Bigr) + O(t^{2})\\
& = (v_{0})_{x}\bigl(\varphi_{0}(y), y\bigr)
+ t\, (v_{0})_{xx}\bigl(\varphi_{0}(y), y\bigr)
\Bigl(\varphi_{1}(y) - \frac{\sqrt{3}}{2}\varphi_{0}(y) + \frac{y - 1}{2}\Bigr)\\
&\qquad + t\, (v_{1})_{x}\bigl(\varphi_{0}(y), y\bigr) + O(t^{2}).
\end{align*}
Substituting \(\varphi_{0}(y) = 0\) and the explicit formula for \(v_{0}\), we obtain
\((v_{0})_{xx}(0, y) = - \frac{3y}{2}\). Solving for \(\varphi_{1}(y)\) yields
\begin{align} 
\varphi_{1}(y) = \frac{2(v_{1})_{x}(0, y)}{3y} - \frac{y}{2} + \frac{1}{2}, 
\qquad y\in(0, 1).
\end{align}
This completes the proof.
\end{proof}

For $t > 0$, the side ${A}_{0}B_{t}$ is strictly longer than the other two sides, and hence, by
Theorem~\ref{thm14LongestSide}, the fail point lies on ${A}_{0}B_{t}$. 
By Theorem~\ref{thm15Uniqueness} and Proposition~\ref{prop41Step2}, the fail point $p_{t} = (x(t), 0)$ satisfies
\begin{align} \label{eqCtm14a}
x(t) = \varphi(t; 0) = t\varphi_{1}(0) + O(t^{2}), 
\end{align}
as $t \to 0^{ + }$, where
\begin{align}\label{eqCtm14bVar1}
\varphi_{1}(0) = \frac{2}{3}(v_{1})_{xy}(0, 0) + \frac{1}{2}.
\end{align}

We are now ready to prove Theorem~\ref{thm18ctm}. By \eqref{eqCtm14a} and \eqref{eqCtm14bVar1}, the key
step is to estimate \((v_{1})_{xy}(O)\), where \(O = (0, 0)\). Since no explicit formula for \(v_{1}\) is available, we use a
barrier argument. The proof is as follows. 

\begin{proof}[Proof of Theorem \ref{thm18ctm}]
Let $w_{1}$ and $w_{2}$ be the unique solutions of 
\begin{align}\label{eqCtm16a}
\begin{cases}
- \Delta w_{1} = \frac{3\sqrt{3}}{2}y & \text{in } \Omega_{0}, \\
w_{1} = 0 & \text{on } \partial\Omega_{0}, 
\end{cases}
\end{align}
and
\begin{align}\label{eqCtm16b}
\begin{cases}
- \Delta w_{2} = \frac{3}{2}x & \text{in } \Omega_{0}, \\
w_{2} = 0 & \text{on } \partial\Omega_{0}. 
\end{cases}
\end{align}
Then $v_{1}(x, y) = w_{1}(x, y) - w_{2}(x, y)$, and hence $(v_{1})_{xy}(x, y) = (w_{1})_{xy}(x, y) - (w_{2})_{xy}(x, y)$. Note that $w_{1}(x, y) = w_{1}( - x, y)$ (by symmetry of $\Omega_{0}$ and uniqueness of solutions), which
implies $(w_{1})_{x}(0, y) = 0$. Consequently, $(w_{1})_{xy}(O) = 0$, and therefore
\begin{align*} 
(v_{1})_{xy}(O) = - (w_{2})_{xy}(O). 
\end{align*}

We next estimate \((w_{2})_{xy}(O)\) by comparing the Taylor expansion of \(w_{2}\) with a suitably
chosen barrier function near the origin. 
By uniqueness of solutions, the anti-symmetry of the right-hand side in the equation for \(w_{2}\) and the symmetry of \(\Omega_{0}\) with respect to the \(y\)-axis, we know that \(w_{2}\) is
anti-symmetric with respect to the \(y\)-axis. In particular, \(w_{2}\) vanishes on the boundary of
the right half-triangle \(\triangle_{OB_{0}C_{0}}\). By the maximum principle, 
\(w_{2} > 0\) in \(\triangle_{OB_{0}C_{0}}\). Since \(w_{2} = 0\) when \(xy = 0\), it follows from the equation for \(w_{2}\) that 
\begin{equation}\label{eqCtm17w2}%
w_{2}(x, y) = xy\big( (w_{2})_{xy}(O) - \frac{3}{4}y + O(x^2 + y^2)\big) \quad \text{as } |x^2 + y^2| \to 0.
\end{equation}
The positivity of \(w_{2}\) in \(\triangle_{OB_{0}C_{0}}\) therefore forces
\begin{equation}
(w_{2})_{xy}(O) > 0.
\end{equation}
%
%
%
%
We choose the barrier function 
\begin{align}\label{eqCtm18barrier}
g(x, y) = \frac{1}{8}xy\Big(\frac{3}{2}(1 - y)^{2} - \frac{9}{2}x^{2} + (1 - y)^{3} - 3\sqrt{3}x^{3}\Big), 
\end{align}
which satisfies
\begin{align*}
\begin{cases}
- \Delta g = \frac{3}{2}x + \frac{3}{2}xy^{2} + \frac{9\sqrt{3}}{2}x^{2}y & \text{in $\triangle_{OB_{0}C_{0}}$}\\
g = 0 & \text{on $\partial \triangle_{OB_{0}C_{0}}$}
\end{cases}
\end{align*}
The factor \(xy\) ensures that \(g\) vanishes on the sides \(OC_{0}\) and \(OB_{0}\). The remaining factor is chosen so that \(g\) also vanishes on \(B_{0}C_{0}\), since \(1 - y = \sqrt{3}x\) for \((x, y) \in B_{0}C_{0}\). Hence 
\begin{equation*}
\begin{cases}
- \Delta (g - w_{2}) = \frac{3}{2}xy^{2} + \frac{9\sqrt{3}}{2}x^{2}y > 0 & \text{in $\triangle_{OB_{0}C_{0}}$}\\
g - w_{2} = 0 & \text{on $\partial \triangle_{OB_{0}C_{0}}$}
\end{cases}
\end{equation*}
%
By the maximum principle, \(g - w_{2} > 0\) in \(\triangle_{OB_{0}C_{0}}\). Using Serrin's boundary lemma at a corner (see, e.g., \cite[Lemma~2]{Ser71} or \cite[Lemma~S]{GNN79}), for $\vec{s} = (1, 1)/\sqrt{2}$, we have 
\begin{equation*}
\text{either} \quad \partial_{\vec{s}}(g - w_{2})(O) > 0 \quad \text{or} \quad \partial_{\vec{s}\vec{s}}(g - w_{2})(O) > 0.
\end{equation*} 
Since $|\nabla (g - w_{2})|(O) = 0$, $(g - w_{2})_{xx}(O) = (g - w_{2})_{yy}(O) = 0$ and \(\partial_{\vec{s}\vec{s}} = (\partial_{xx} + 2\partial_{xy} + \partial_{yy})/2\), we must have $\partial_{\vec{s}}(g - w_{2})(O) = 0$, and hence
\begin{equation}
\partial_{xy}(g - w_{2})(O) = \partial_{\vec{s}\vec{s}}(O) > 0.
\end{equation}
Therefore, \((w_{2})_{xy}(O) < g_{xy}(O) = 5/16\), and hence
\begin{align} \label{eqCtm20}
0 < (w_{2})_{xy}(O) < \frac{5}{16}.
\end{align}
Combining this with $(v_{1})_{xy}(O) = - (w_{2})_{xy}(O)$, we obtain 
\begin{equation}\label{eqCtm21a}
- \frac{5}{16} < (v_{1})_{xy}(O) < 0.
\end{equation}
It follows that $\partial_{t}\varphi(0; 0) = \varphi_{1}(0) = (2/3)(v_{1})_{xy}(O) + 1/2$ satisfies 
\begin{equation*}
\frac{7}{24} < \partial_{t}\varphi(0; 0) < \frac{1}{2}.
\end{equation*}
Consequently, 
\begin{equation*} 
\frac{7}{24}t < x(t) = \varphi(t; 0) < \frac{1}{2} t 
\end{equation*}
for \(t > 0\) sufficiently small. This completes the proof of the estimate \eqref{estimatedinterval} for
the location of the fail point.

Next, we prove~\eqref{compare1}.
By~\eqref{eqCtm11Finv}, a direct computation shows that
\begin{equation*}
F_{t}^{ - 1}(x, y) = \Big(x + t\big(\frac{y - 1}{2} - \frac{\sqrt{3}}{2}x\big) + t^2\big(\frac{3}{4}x - \frac{\sqrt{3}}{4}(y - 1)\big) + O(t^3), y\Big).
\end{equation*}
Since $u(t; x, y) = v(t; F_{t}^{ - 1}(x, y))$, we have 
\begin{align} \label{eqCtm23cangzhi1}
u_{y} (t; x, y) = v_{x} \left(t; F_{t}^{ - 1}(x, y)\right)\Big(\frac{t}{2} - \frac{\sqrt{3}}{4}t^2\Big) + v_{y} \left(t; F_{t}^{ - 1}(x, y)\right) + O(t^3), 
\quad t \to 0.
\end{align}
Recalling that $v(t; x, y) = v_{0}(x, y) + tv_{1}(x, y) + t^2 v_{2}(x, y) + O(t^3)$, the explicit formula for \(v_{0}\), Taylor expansion, and straightforward computations using~\eqref{eqCtm23cangzhi1} yield
\begin{align*}
u_{y} (t; \frac{t}{2}, 0) - u_{y} (t; 0, 0) = \frac{t^2}{2}\left((v_{1})_{xy}(O) + \frac{3}{8}\right) + O(t^3).
\end{align*}
Combining this with the key estimate~\eqref{eqCtm21a}, we obtain
\begin{align*}
u_{y} (t; \frac{t}{2}, 0) - u_{y} (t; 0, 0) > \frac{1}{32}t^2, 
\end{align*}
provided \(0 < |t| \ll 1\). This completes the proof.
\end{proof}

\begin{remark}
The barrier function given by \eqref{eqCtm18barrier} was constructed after many trials, and the interval estimation \eqref{estimatedinterval} of the fail point obtained through this function is currently the most refined version we have achieved.
\end{remark}

We conclude this section by considering another perturbation of the equilateral triangle. We obtain the following result on the gradient of the torsion function along the second-longest side.

\begin{theorem} \label{thm18vertical}
Let $\{\Omega_{t}\}_{t \geq 0}$ be the family of triangles with vertices $A_{t} = ( - \sqrt{3}/3, 0)$, $B_{t} = (\sqrt{3}/3, 0)$ and $C_{t} = (t, 1)$, and let $u(t; x, y)$ denote the torsion function on $\Omega_{t}$. 
Then, for $t \to 0^{ + }$, the (unique) critical point of $x \mapsto \lvert \nabla u(t; x, 0)\rvert^{2}$ is given by $x(t)$, where 
\begin{align}\label{estimatedinterval2}
0 < x(t) < \frac{5}{12}t.
\end{align}
Moreover, 
\begin{align} \label{compare2}
u_{y} (t; 0, 0) - u_{y} (t; t, 0) > \frac{1}{8}t^2.
\end{align}
\end{theorem}

\begin{rmks}
The triangle above is constructed by moving the top vertex of an equilateral triangle to the right while fixing the base side. On the base side, the midpoint is given by $(0, 0)$ and the foot of the altitude is given by $(0, t)$, and thus according to \eqref{estimatedinterval2}, $|x(t) - 0| < |x(t) - t|$ for $t > 0$ being small. Therefore, the theorem above also shows that restricted to the base side, the critical point of the norm of the gradient of the torsion function is closer to the midpoint. Also, \eqref{compare2} implies that the norm of the gradient has a larger value at the midpoint.
\end{rmks}

To prove Theorem~\ref{thm18vertical}, we still let \(\mathcal{N}_{t}\) denote the nodal line of \(u_{x} (t; x, y)\). As above, for \(t > 0\) sufficiently small, \(\mathcal N_{t}\) can be written as the graph of a function of \(y\), which we again denote by \(\varphi\). 
Define 
\begin{equation*}
F_{t}(x, y) = (x + ty, y) \quad \text{ and } \quad v(t; x, y) = u(t; F_{t}(x, y)).
\end{equation*}
The map \(F_{t}\) is linear and sends \(\Omega_{0}\) onto \(\Omega_{t}\). One checks that \(v(t; \cdot)\) satisfies 
\begin{align} 
\begin{cases}
(1 + t^{2})v_{xx} + v_{yy} - 2tv_{xy} = - 1 & \text{ in } \triangle ABC_{0}, \\
v = 0 & \text{ on } \partial \triangle ABC_{0}.
\end{cases}
\end{align}
By standard elliptic regularity and the implicit function theorem, the map $t \mapsto v(t; \cdot)$ is smooth from $[0, 1]$ to $C^{1}(\overline{\Omega_{0}})$; in particular, $v(t; x, y)$ depends smoothly on $t$. Assume that $v(t; x, y) = v_{0}(x, y) + tv_{1}(x, y) + t^2 v_{2}(x, y) + O(t^3)$ as $t \to 0$. Then \(v_{0}\) is the torsion function on \(\Omega_{0}\), and \(v_{1}\) satisfies 
\begin{align*}
\begin{cases}
\Delta v_{1} = 2(v_{0})_{xy} & \text{in } \Omega_{0}, \\
v_{1} = 0 & \text{on } \partial\Omega_{0}. 
\end{cases}
\end{align*}
Recalling that \(v_{0}\) is given explicitly by \eqref{eqCtm07v0}, we have \(2(v_{0})_{xy}(x, y) = - 3x\). Therefore \(v_{1}\) is determined by 
\begin{align*}
\begin{cases}
- \Delta v_{1} = 3x & \text{in } \Omega_{0}, \\
v_{1} = 0 & \text{on } \partial\Omega_{0}. 
\end{cases}
\end{align*}

With these preparations, the proof of Theorem \ref{thm18vertical} goes as follows.

\begin{proof}[Proof of Theorem \ref{thm18vertical}]
By the definition of \(\varphi\), we have \(u_{x}(t; \varphi(t; y), y) = 0\) and \(\varphi(0; y) = 0\). Thus we
may expand
\begin{align}\label{eqCtm32}
\varphi(t; y) = t\varphi_{1}(y) + O(t^{2}). 
\end{align}
Recalling that \(v_{x}(t; \varphi(t; y) - ty, y) = 0\), we compute
\begin{align*}
& \quad\; v_{x}(t; \varphi(t; y) - ty, y) 
\\
& = (v_{0})_{x}(0, y) + t\Big((v_{0})_{xx}(0, y)(\varphi_{1}(y) - y) + (v_{1})_{x}(0, y)\Big) + O(t^{2})
\\ 
& = t\big( - \frac{3}{2}y(\varphi_{1}(y) - y) + (v_{1})_{x}(0, y)\big) + O(t^{2}), 
\end{align*}
where we used \((v_{0})_{x}(0, y) = 0\) and \((v_{0})_{xx}(0, y) = - 3y/2\) (see \eqref{eqCtm07v0xx}). This yields
\begin{align} \label{eqCtm33varphi1}
\varphi_{1}(y) = \frac{2}{3}\frac{(v_{1})_{x}(0, y)}{y} + y \text{ and }
\varphi_{1}(0) = \frac{2}{3}(v_{1})_{xy}(O).
\end{align}
Clearly, \(v_{1} = 2w_{2}\), where \(w_{2}\) is given in \eqref{eqCtm16b} and satisfies
\eqref{eqCtm20}. Therefore, 
\begin{align} \label{eqCtm34UpperBoundV1xy}
0 < (v_{1})_{xy}(O) < \frac{5}{8}.
\end{align}
Combining \eqref{eqCtm32}, \eqref{eqCtm33varphi1}, and \eqref{eqCtm34UpperBoundV1xy}, we obtain
\begin{equation*}
0 < \varphi(t; 0) < \frac{5}{12}t, 
\end{equation*}
for \(t > 0\) sufficiently small. This completes the proof of \eqref{estimatedinterval2}.

Next, we prove~\eqref{compare2}. 
Since \(u(t; x, y) = v(t; x - ty, y)\), we have
\begin{align} \label{eqCtm35cangzhi2}
u_{y} (t; x, y) = - tv_{x} (t; x - ty, y) + v_{y} (t; x - ty, y).
\end{align}
Recalling the expansion $v(t; x, y) = v_{0}(x, y) + tv_{1}(x, y) + t^2 v_{2}(x, y) + O(t^3)$, it follows from~\eqref{eqCtm35cangzhi2}, the explicit form of \(v_{0}\), and a Taylor expansion that 
\begin{align*}
u_{y} (t; 0, 0) - u_{y} (t; t, 0) = \Big(\frac{3}{4} - (v_{1})_{xy}(O)\Big)t^2 + O(t^3), \quad t \to 0. 
\end{align*}
By~\eqref{eqCtm34UpperBoundV1xy}, we conclude that
\begin{align*}
u_{y} (t; 0, 0) - u_{y} (t; t, 0) > \big(\frac{3}{4} - \frac{5}{8}\big)t^2 = \frac{1}{8}t^2, 
\end{align*}
provided \(0 < |t| \ll 1\). 
This completes the proof.
\end{proof} 


\section{Fail points on non-concentric annuli} \label{Sect5annuli}

In this section, we consider a class of multiply connected domains, namely annuli. Concentric annuli possess radial symmetry, with every point on each boundary equidistant from a common center, which facilitates straightforward mathematical analysis. In contrast, non-concentric annuli lack such symmetry, resulting in more varied and complex geometries due to varying radial distances from distinct centers. We study the location of fail points in annuli using the reflection method. 

\begin{proof}[Proof of Theorem \ref{thm19Annulus}]
Without loss of generality, we assume that the non-concentric annulus $\Omega$ is given by 
\begin{equation}
\Omega = \{x \in \R^{2}: \; |x| < \rho_{1}, \; |x - {I}| > \rho_{2}\}, 
\end{equation}
where ${I} = (\epsilon, 0)$ and $\rho_{1}$, $\rho_{2}$, and $\epsilon$ are positive constants satisfying $\rho_{2} + \epsilon < \rho_{1}$.

\textbf{Step 1}. 
We first show that fail points must be located on the inner boundary, by proving that for every outer boundary point $p$, there exists a corresponding inner boundary point $q$ such that $|\nabla u(p)| < |\nabla u(q)|$. 
In fact, fix any outer boundary point ${p}^{\theta} = (\rho_{1}\cos\theta, \rho_{1}\sin\theta)$. Let ${V}^{\theta} = (\rho\cos\theta, \rho\sin\theta)$, where $\rho = \rho_{1} + \rho_{2}$. Then ${p}^{\theta}$ is the unique intersection point of two circles $\partial{B}_{\rho_{1}}$ and $\partial{B}_{\rho_{2}}({V}^{\theta})$, see Figure \ref{fig41}. Let ${T}^{\theta}$ denote the perpendicular bisector of the line segment connecting ${V}^{\theta}$ and the center ${I} = (\epsilon, 0)$ of the inner disk, that is, 
\begin{equation}
{T}^{\theta} = \{x \in \R^{2}: 2x_{1}(\rho\cos\theta - \epsilon) + 2x_{2}\rho\sin\theta = \rho^2 - \epsilon^2\}.
\end{equation}

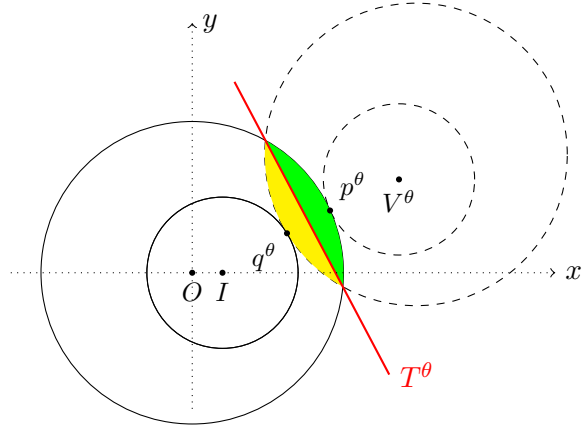
\begin{figure}[htp]
\centering
\begin{tikzpicture}[scale = 2]
\pgfmathsetmacro\radiO{1}; 
\pgfmathsetmacro\radiI{\radiO*0.5}; 
\pgfmathsetmacro\xI{0.4*(\radiO - \radiI)}; \pgfmathsetmacro\yI{0}; 
\pgfmathsetmacro\Jiao{31.44}; 
\pgfmathsetmacro\xC{\xI + (\radiO + \radiI)*cos(\Jiao)}; 
\pgfmathsetmacro\yC{\yI + (\radiO + \radiI)*sin(\Jiao)}; 
\pgfmathsetmacro\xQ{\xI + \radiI*cos(\Jiao)}; 
\pgfmathsetmacro\yQ{\yI + \radiI*sin(\Jiao)}; 
\pgfmathsetmacro\xP{\xC - \xQ*(\xC^2 - \yC^2)/(\xC^2 + \yC^2) - 2*\yQ*\xC*\yC/(\xC^2 + \yC^2)}; 
\pgfmathsetmacro\yP{\yC + \yQ*(\xC^2 - \yC^2)/(\xC^2 + \yC^2) - 2*\xQ*\xC*\yC/(\xC^2 + \yC^2)}; 
\pgfmathsetmacro\xD{\xP*(\radiO + \radiI)/\radiO}; 
\pgfmathsetmacro\yD{\yP*(\radiO + \radiI)/\radiO}; 
\draw (0, 0) circle (\radiO); 
\draw (\xI, \yI) circle (\radiI); 
\draw[dashed] (\xC, \yC) circle (\radiO); 
\draw[dashed] (\xD, \yD) circle (\radiI); 
\pgfmathsetmacro\xM{\xC/2 + \yC*sqrt(\radiO^2/(\xC^2 + \yC^2) - 1/4)}; 
\pgfmathsetmacro\yM{\yC/2 - \xC*sqrt(\radiO^2/(\xC^2 + \yC^2) - 1/4)}; 
\pgfmathsetmacro\xN{\xC/2 - \yC*sqrt(\radiO^2/(\xC^2 + \yC^2) - 1/4)}; 
\pgfmathsetmacro\yN{\yC/2 + \xC*sqrt(\radiO^2/(\xC^2 + \yC^2) - 1/4)}; 
\fill[gray, green] (\xM, \yM) arc ({asin(\yM/\radiO)} : {asin(\yN/\radiO)} : \radiO) -- cycle; 
\fill[gray, yellow] (\xM, \yM) arc ({ - acos((\xM - \xC)/\radiO)} : { - 180 - asin((\yN - \yC)/\radiO)} : \radiO) -- cycle; 
\draw[thick, red] ({\xM + 1.4*(\xN - \xM)}, {\yM + 1.4*(\yN - \yM)}) -- ({\xM - 0.6*(\xN - \xM)}, {\yM - 0.6*(\yN - \yM)}) node[right] {${T}^{\theta}$}; 
\draw (\xI, \yI) circle (\radiI); 
\fill (\xI, \yI) circle (0.02) node[below] {\small ${I}$}; 
\fill (0, 0) circle (0.02) node[below] {\small ${O}$}; 
\fill (\xD, \yD) circle (0.02) node[below] {\small ${V}^{\theta}$}; 
\fill (\xQ, \yQ) circle (0.02) node[below left] {\small ${q}^{\theta}$}; 
\fill (\xP, \yP) circle (0.02) node[above right] {\small ${p}^{\theta}$}; 
\draw[dotted, -> ] ( - 1.2*\radiO, 0) -- (2.4*\radiO, 0) node[right]{$x$}; 
\draw[dotted, -> ] (0, - 0.9*\radiO) -- (0, 1.65*\radiO, 0) node[right]{$y$}; 
\end{tikzpicture}
\caption{Non-concentric annulus: inner and outer boundaries}
\label{fig41}
\end{figure}

Let ${q}^{\theta} = ({q}_{1}^{\theta}, {q}_{2}^{\theta})$ denote the reflection of the point ${p}^{\theta}$ with respect to ${T}_{\theta}$, where 
\begin{align*}
{q}_{1}^{\theta} & = \frac{(\rho^2 - \epsilon^2)(\rho\cos\theta - \epsilon) - \rho_{1}(\rho^2\cos\theta - 2\rho\epsilon + \epsilon^2\cos\theta)}{\rho^2 + \epsilon^2 - 2\rho\epsilon\cos\theta}, 
\\
{q}_{2}^{\theta} & = \frac{(\rho^2 - \epsilon^2)(\rho - \rho_{1})\sin\theta}{\rho^2 + \epsilon^2 - 2\rho\epsilon\cos\theta}.
\end{align*}
By construction, ${q}^{\theta}$ lies on the inner boundary. It is straightforward to verify that 
\begin{equation}
u(x) - u^{\theta} (x) < 0 \text{ for } x \in {D}^{\theta}, 
\end{equation}
where $u^{\theta} (x): = u ({T}^{\theta}x)$, with ${T}^{\theta}x$ denoting the reflection of $x$ with respect to ${T}_{\theta}$, and ${D}^{\theta}$ is the smaller cap cut by ${T}^{\theta}$ from ${B}_{\rho_{1}}$. By the Hopf lemma, it follows that $x \cdot \nabla (u - u^{\theta}) > 0$ at $x = {p}^{\theta}$. Thus, $|\nabla u({p}^{\theta})| < |\nabla u({q}^{\theta})|$. 

\textbf{Step 2}. 
Next, we determine the precise location of the fail point on the inner boundary. 
Let ${L}_{\varphi}$ be the straight line passing through the center of the inner circle with inclination $\varphi\in(0, \pi)$, that is, ${L}_{\varphi} = \{ x \in \Omega: (x_{1} - \epsilon)\sin\varphi + x_{2}\cos\varphi = 0\}$. Define
\begin{equation*}
\tilde{D}_{\varphi} = \{ x \in \Omega: (x_{1} - \epsilon)\sin\varphi + x_{2}\cos\varphi > 0\}.
\end{equation*}

\begin{figure}[htp]
\centering
\begin{tikzpicture}[scale = 2]
\pgfmathsetmacro\radiO{1}; 
\pgfmathsetmacro\radiI{\radiO*0.5}; 
\pgfmathsetmacro\xI{0.4*(\radiO - \radiI)}; \pgfmathsetmacro\yI{0}; 
\pgfmathsetmacro\Jiaoa{63.44}; 
\pgfmathsetmacro\Jiaob{(180 + \Jiaoa)/2}; 
\pgfmathsetmacro\yM{ - \xI*sin(\Jiaob)*cos(\Jiaob) + sqrt(\radiO^2*(sin(\Jiaob))^2 - \xI^2*(sin(\Jiaob))^4)}; 
\pgfmathsetmacro\yN{ - \xI*sin(\Jiaob)*cos(\Jiaob) - sqrt(\radiO^2*(sin(\Jiaob))^2 - \xI^2*(sin(\Jiaob))^4)}; 
\pgfmathsetmacro\xM{\xI + \yM/tan(\Jiaob)}; 
\pgfmathsetmacro\xN{\xI + \yN/tan(\Jiaob)}; 
\fill[gray, green, draw = black] (\xM, \yM) arc ({acos(\xM/\radiO)} : { - acos(\xN/\radiO)} : \radiO) -- ({\xI - \radiI*cos(\Jiaob)}, {\yI - \radiI*sin(\Jiaob)}) arc (\Jiaob - 180 : \Jiaob : \radiI) -- cycle; 
\fill[gray, yellow, draw = black, dashed] (\xM, \yM) arc ({2*\Jiaob - acos(\xM/\radiO)} : {2*\Jiaob + acos(\xN/\radiO)} : \radiO) -- ({\xI - \radiI*cos(\Jiaob)}, {\yI - \radiI*sin(\Jiaob)}) arc (\Jiaob + 180 : \Jiaob : \radiI) -- cycle; 
\draw[thick, red] (\xI, \yI) -- ++ (\Jiaob : \radiO*1.4); 
\draw[thick, red] (\xI, \yI) -- ++ (\Jiaob : - \radiO*1.1) node[above = 4pt, right] {${L}_{\varphi}$}; 
\draw[dotted, -> ] ( - 1.2*\radiO, 0) -- (1.5*\radiO, 0) node[right]{$x$}; 
\draw[dotted, -> ] (0, - 1.1*\radiO) -- (0, 1.3*\radiO, 0) node[right]{$y$}; 
\draw (0, 0) circle (\radiO); 
\draw (\xI, \yI) circle (\radiI); 
\fill (\xI, \yI) circle (0.02) node[below] {\small ${I}$}; 
\fill (0, 0) circle (0.02) node[below] {\small ${O}$}; 
\fill (\xI - \radiI, 0) circle (0.02) node[right] {\small ${q_{0}}$}; 
\fill ({\xI + \radiI*cos(\Jiaoa)}, {\radiI*sin(\Jiaoa)}) circle (0.02) node[below] {\small ${q_{\varphi}}$}; 
\end{tikzpicture}
\caption{Non-concentric annulus: location of fail point}
\label{fig42}
\end{figure}
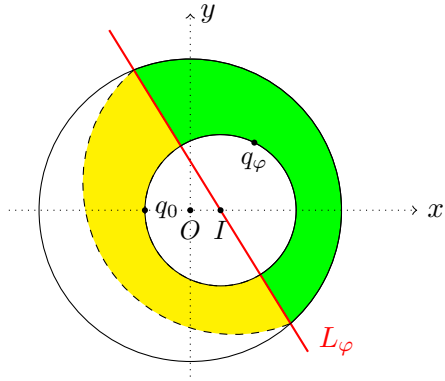

Once again, one can show that
\begin{equation}
u(x) - u_{\varphi} (x) < 0 \text{ for } x \in \tilde{D}_{\varphi}, 
\end{equation}
where $u_{\varphi} (x): = u ({L}_{\varphi}x)$, ${L}_{\varphi}x$ denoting the reflection of $x$ with respect to ${L}_{\varphi}$, and $\tilde{D}_{\varphi}$ is the smaller region cut by ${L}_{\varphi}$ from $\Omega$, see Figure \ref{fig42}. The Hopf lemma implies that $(x - {I}) \cdot \nabla (u - u_{\varphi}) < 0$ on $(\partial\tilde{D}_{\varphi} \setminus {L}_{\varphi}) \cap \partial{B}_{\rho_{2}}({I})$. In particular, 
\begin{equation*}
|\nabla u({q}_{0})| > |\nabla u({q}_{\varphi})| \text{ for } \varphi\in(0, \pi), 
\end{equation*}
where ${q}_{\varphi} = (\epsilon - \rho_{2}\cos(2\varphi), - \rho_{2}\sin(2\varphi))$. This completes the proof of Theorem \ref{thm19Annulus}.
\end{proof}

As a consequence of the theorem above, we have the following remarks.

\begin{remark}
An annulus has a unique fail point if and only if it is non-concentric. 
\end{remark}

\begin{remark}
The conclusion of Theorem \ref{thm19Annulus} remains valid for non-concentric annuli in higher dimensions. Moreover, by employing the moving plane method instead of the reflection argument, one can show that the maximum of $|\nabla u|$ on $\partial\Omega$ is uniquely attained at the point on the inner boundary closest to the center of the outer boundary, where $u$ denotes a positive solution to the semilinear equation \eqref{eq307sl} in non-concentric annuli $\Omega$. 
\end{remark}


\section*{Acknowledgments}
Research of Qinfeng Li was supported by National Natural Science Fund of China for Excellent Young Scholars (No. 12522109), National Key R\&D Program of China (No. 2022YFA1006900) and the National Science Fund of China General Program (No. 12471105). 
Research of Shuangquan Xie was supported by the Changsha Natural Science Foundation (No. KQ2208006). 
Research of Ruofei Yao was supported by Guangdong Basic and Applied Basic Research Foundation (Grant No. 2025A1515011856). 

\bibliographystyle{plain} 
\bibliography{BibLiQinfengYaoRuofei}

\end{document}